\documentclass[a4paper]{amsart}
\usepackage{amsmath,amsthm,amssymb,mathtools,latexsym,epic,bbm,comment}
\usepackage{graphicx,enumerate,stmaryrd,tabularx,youngtab}
\usepackage[all,2cell]{xy}
\xyoption{2cell}

\newtheorem{theorem}{Theorem}
\newtheorem{lemma}[theorem]{Lemma}
\newtheorem{corollary}[theorem]{Corollary}
\newtheorem{proposition}[theorem]{Proposition}
\newtheorem{conjecture}[theorem]{Conjecture}
\newtheorem{definition}[theorem]{Definition}

\theoremstyle{definition}
\newtheorem{remark}[theorem]{Remark}

\usepackage[all]{xy}
\usepackage[active]{srcltx}
\usepackage[parfill]{parskip}
\usepackage{enumerate}

\newcommand{\tens}[1]{\mathbin{\mathop{\otimes}\displaylimits_{#1}}}

\newcommand{\quotient}[2]{\raisebox{.15em}{$#1$}\left/\raisebox{-.15em}{$#2$}\right.}

\usepackage{tikz}
\usetikzlibrary{fit,matrix,positioning,calc,cd}

\newcommand{\C}{\mathbb{C}}

\newcommand{\Z}{\mathbb{Z}}
\newcommand{\LL}{\mathfrak{L}}

\newcommand{\g}{\mathfrak{g}}
\newcommand{\h}{\mathfrak{h}}
\newcommand{\n}{\mathfrak{n}}
\newcommand{\bb}{\mathfrak{b}}
\newcommand{\rr}{\mathfrak{r}}

\newcommand{\F}{\mathcal{F}}
\renewcommand{\O}{\mathcal{O}}
\DeclareMathOperator{\Hom}{Hom}
\DeclareMathOperator{\End}{End}
\DeclareMathOperator{\Ext}{Ext}
\DeclareMathOperator{\Ker}{Ker}

\DeclareMathOperator{\Ind}{Ind}
\DeclareMathOperator{\Rad}{Rad}

\DeclareMathOperator{\ad}{ad}
\DeclareMathOperator{\Sym}{Sym}
\DeclareMathOperator{\Span}{span}

\DeclareMathOperator{\EnAr}{\mathbf{EA}}

\newcommand{\HC}{Harish-Chandra}

\let\originalleft\left
\let\originalright\right
\renewcommand{\left}{\mathopen{}\mathclose\bgroup\originalleft}
\renewcommand{\right}{\aftergroup\egroup\originalright}

\calclayout

\begin{document}

\title[$\mathfrak{sl}_2$-Harish-Chandra modules for $\mathfrak{sl}_2 \ltimes L(4)$]
{$\mathfrak{sl}_2$-Harish-Chandra modules for $\mathfrak{sl}_2 \ltimes L(4)$}

\author{Volodymyr Mazorchuk and Rafael Mr{\dj}en}

\begin{abstract}
We use analogues of Enright's and Arkhipov's functors to determine the quiver and 
relations for a category of $\mathfrak{sl}_2 \ltimes L(4)$-modules which are locally 
finite (and with finite multiplicities) over $\mathfrak{sl}_2$. We also outline
serious obstacles to extend our result to $\mathfrak{sl}_2 \ltimes L(k)$,
for $k>4$.
\end{abstract}

\maketitle


\section{Introduction and description of the results}\label{s1}

We work over the field $\mathbb{C}$ of complex numbers.
For a non-negative integer $n\in \Z_{\geq 0}$, denote by $L(n)$ the unique simple module
of dimension $n+1$ over the Lie algebra $\mathfrak{sl}_2$
(the module $L(n)$ has highest weight $n$).
In this paper, we study certain categories of modules over the semi-direct product
\[ \LL^n := \mathfrak{sl}_2 \ltimes L(n), \]
where  $\rr := L(n)$ is an abelian ideal.
The algebra $\LL^n$ is called a \emph{conformal Galilei algebra}
(for a more general definition and various central extensions, 
see \cite{lu2014simple, alshammari2019on, gomis2012schrodinger}).

In our previous paper \cite{mazorchuk2020lie}, for a finite dimensional
Lie algebra $\LL$ with a fixed Levi decomposition 
$\mathfrak{g}\ltimes \rr$, where $\mathfrak{g}$ is semi-simple,
we looked at the category of $\LL$-modules which are locally 
finite and with finite multiplicities over the universal enveloping
algebra of $\mathfrak{g}$.
We called such modules \emph{$\mathfrak{g}$-Harish-Chandra modules}.
The two main examples considered in \cite{mazorchuk2020lie} were:
\begin{itemize}
\item a central extension of the algebra $\LL^1$, known as the
\emph{Schr\"{o}dinger Lie algebra},
\item the algebra $\LL^2$, known as the
\emph{Takiff $\mathfrak{sl}_2$}.
\end{itemize}
We note that $\LL^0 \cong \mathfrak{gl}_2$, which has fairly trivial
$\g$-Harish-Chandra theory. Similarly, 
the \emph{centerless Schr\"{o}dinger Lie algebra} $\LL^1$
also has trivial $\g$-Harish-Chandra theory, cf. \cite[Remark~54]{mazorchuk2020lie}.
In \cite[Section~4]{mazorchuk2020lie} we obtain a number of interesting results
about $\g$-Harish-Chandra modules for $\LL^2$. This includes a classification of 
simple $\g$-Harish-Chandra modules for $\LL^2$, see \cite[Theorem~30]{mazorchuk2020lie}
and the claim that there are no non-trivial extensions between different infinite-dimensional 
simple $\g$-\HC{} modules and that self-extensions of such modules are essentially unique
and do not have radical central character, cf. \cite[Theorem~31]{mazorchuk2020lie}.
These results were, in part, inspired by the study of the category $\mathcal{O}$
for $\LL^2$ in \cite{mazorchuk2019category}. The connection between the two
categories is given in terms of certain adaptations of the classical 
Enright's and Arkhipov's functors.

The category $\mathcal{O}$ for the Schr\"{o}dinger Lie algebra was studied 
in \cite{dubsky2014category}, which, similarly, served as an inspiration for the study of 
the corresponding category of $\g$-Harish-Chandra modules in 
\cite[Section~5]{mazorchuk2020lie}.

The main reason why the Takiff $\mathfrak{sl}_2$ and the Schr\"{o}dinger Lie algebra allowed 
for a full description of the category of $\g$-\HC{} modules in \cite{mazorchuk2020lie} 
is that the centers of their universal enveloping algebras were easy to write down. 
Moreover, in both cases there is a central element which does not belong to the purely 
radical part of the center, which prevented extensions between different $\g$-\HC{} modules. 
We suspect that this is not the case for higher conformal Galilei algebras, i.e. that the 
center of $U(\LL^n)$ consists only of the radical part, for all $n \geq 3$. An explicit 
description of this radical part of the center remains a hard (classical) 
invariant-theoretic problem. In Section \ref{sec:gen_funct} we use generating functions, 
to obtain some interesting information about this problem for $n\leq 6$. 
The results in Section~\ref{sec:gen_funct} suggest that $\LL^4$ is the biggest conformal 
Galilei algebra that allows for a reasonable description of the category of $\g$-\HC{} 
modules (cf. Subsection~\ref{sub:L5} and \ref{sub:L6}).

Taking the above into account, the present paper is a natural continuation of
\cite{mazorchuk2020lie}. We try to understand $\g$-\HC{} modules for $\LL^4$. 
In particular,  in Section~\ref{section:existente_gHC} we describe the 
structure of the category of finitely generated $\g$-\HC{} modules with a certain 
restriction on the (radical part of the) central character, namely that it lies on 
the curve $x^3 = y^2$ in the complex plane $\C^2$. Our main result is 
Theorem~\ref{thm:quiver}. This theorem describes each blocks of this category via 
the corresponding Gabriel quiver and relations. 
Similarly to \cite{mazorchuk2020lie}, our main tool is 
an analogue of Enright's and Arkhipov's functors which connects our category 
with the category $\mathcal{O}$ for $\LL^4$.

In the case of a generic radical central character for $\LL^4$, as well as for the Lie
algebra $\LL^3$, we suspect that the corresponding category of $\g$-\HC{} modules is trivial, 
i.e. that the universal induced modules are already simple. Unfortunately, these modules 
cannot be obtained by localizing the highest weight modules, which is most likely related 
to the failure of Duflo's theorem on annihilators of simple modules for these Lie algebras.
The upshot of all these remarks is that the results of the present paper 
together with those of \cite{mazorchuk2020lie} seem to cover all the categories of 
$\g$-\HC{} modules over conformal Galilei algebra which are possible to 
understand via category $\mathcal{O}$.

Throughout the paper we use results from Section~\ref{sec:gen_funct}. That section contains 
technical combinatorial calculations, which are independent of the previous sections. 
The penultimate section of the paper, Section~\ref{s-new}, contains a crucial (but lengthy) 
combinatorial argument needed in the proof of one of the auxiliary statements 
(Lemma~\ref{lem:ker}) which pave the way for the main Theorem~\ref{thm:quiver}.

\subsection*{Acknowledgments}
This research was partially supported by
the Swedish Research Council, G{\"o}ran Gustafsson Stiftelse and Vergstiftelsen. 
R. M. was also partially supported by the QuantiXLie Center of Excellence grant 
no. KK.01.1.1.01.0004 funded by the European Regional Development Fund.

The authors are grateful to Christian Krattenthaler and Stephan Wagner for their 
help with combinatorial computations, in particular, for the proofs of
Proposition~\ref{proposition:F4} and 
Proposition~\ref{proposition:F5}, respectively.

\section{Notation and preliminaries}\label{s2} 

For a Lie algebra $\mathfrak{a}$, we denote by $U(\mathfrak{a})$ the universal 
enveloping algebra of $\mathfrak{a}$. Tensor products which do not come with a 
subscript are considered to be over $\C$.

\subsection{$\g$-Harish-Chandra modules}

Fix a finite-dimensional semi-simple Lie algebra $\LL$ over $\C$, and fix its Levi decomposition $\LL \cong \g \ltimes \rr$. This is a semi-direct product, where $\g$ is a maximal semi-simple Lie subalgebra, unique up to conjugation, and $\rr = \Rad \LL$ is the radical of $\LL$ , i.e., the unique maximal solvable ideal.

An $\LL$-module $M$ is \emph{$\g$-locally finite}, if any element of $M$ is contained in a finite-dimensional $\g$-submodule of $M$. The \emph{multiplicity} of a simple finite-dimensional $\g$-module $L$ in a $\g$-locally finite $\LL$-module $M$, i.e., $\dim\Hom_\g(L,M)$, is denoted by $[M\colon L]$.

\begin{definition}
An $\LL$-module $M$ is called a \emph{$\g$-Harish-Chandra} module, if it is $\g$-locally finite, and $[M \colon L]$ is finite for any simple finite-dimensional $\g$-module $L$.
\end{definition}

A simple $\g$-submodule of a $\g$-Harish-Chandra module $M$ is called a \emph{$\g$-type} of $M$. The sum of all $\g$-submodules of $M$ isomorphic to a given $\g$-type is called the \emph{$\g$-isotypic component} of $M$ determined by this $\g$-type.

Fix a Cartan subalgebra $\h \subseteq \g$ and a positive root system $\Delta^+(\g,\h)$, and a non-degenerate symmetric invariant bilinear form $\langle -,- \rangle$ on $\h^\ast$. We say that a weight $\lambda \in \h^\ast$ is \emph{dominant} if $\langle
\lambda,\alpha \rangle \geq 0$ for all positive roots $\alpha$, and \emph{integral} if $\langle
\lambda,\alpha \rangle \in \mathbb{Z}$ for all positive roots $\alpha$. For $\lambda \in \h^\ast$ we denote by $L(\lambda)$ the unique simple $\g$-module with highest weight $\lambda$. The module $L(\lambda)$ is finite-dimensional if and only if $\lambda$ is dominant and integral (and therefore automatically regular with respect to the dot-action), cf. \cite{humphreys1978introduction}.

The standard basis elements for $\mathfrak{sl}_2$ will be denoted, as usual, by $f$, $h$, $e$. We denote by $v_n, v_{n-2}, \ldots, v_{-n}$ a basis of the finite-dimensional $\mathfrak{sl}_2$-module $L(n)$, $n \geq 0$, such that each $v_i$ is a weight vector of weight $i$, and $[e,v_{n-2i}] = (n-i+1) v_{n-2i+2}$, for $i \in \{1,2,\ldots,n\}$, cf. \cite{mazorchuk2010lectures, humphreys1978introduction}.

\subsection{On existence of simple infinite-dimensional $\g$-Harish-Chandra modules}

\label{section:existente_gHC}

Let us recall the setup from \cite[Subsection~6.2]{mazorchuk2020lie}.
Fix a finite-dimensional Lie algebra $\LL$ with a Levi decomposition 
$\LL \cong \g \ltimes \rr$ and assume that
\[\rr \text{ is abelian}.\]
Such a Lie algebra is called a \emph{generalized Takiff Lie algebra}.
We consider the {\em purely radical} part $\overline{Z(\LL)}:=Z(\LL)\cap U(\rr)$
of the center $Z(\LL)$ of $U(\LL)$. Since $\rr$ is assumed to be abelian, it is obvious that $\overline{Z(\LL)} = U(\rr)^\g$, the $\g$-invariants in $U(\rr) \cong \Sym(\rr)$ with respect to the adjoint action.

For brevity, algebra homomorphisms 
$\chi:\overline{Z(\LL)}\to \mathbb{C}$ will be called {\em radical 
central characters}. We say that an $\LL$-module $M$ has radical 
central characters $\chi$, if each element $z$ of the purely radical part of the center acts on $M$ as the scalar $\chi(z)$.

For $\lambda \in \h^\ast$ dominant and integral, i.e., such that $L(\lambda)$ is finite-dimensional, we have the  \emph{universal module} 
\[ \mathbf{Q}(\lambda) := \Ind_{\g}^{\LL} L(\lambda) = U(\LL) \tens{U(\g)} L(\lambda) \cong U(\rr) \tens{\C} L(\lambda) \cong \Sym(\rr) \tens{\C} L(\lambda). \]
Clearly, it is $\g$-locally finite. Analogously to \cite[Proposition 6]{mazorchuk2020lie}, one can show that \[ \End \mathbf{Q}(0) \cong \overline{Z(\LL)}. \]
Given also a radical central character $\chi \colon \overline{Z(\LL)} \to \C$, we define
\[Q(\lambda,\chi):=\quotient{\mathbf{Q}(\lambda)}{\mathbf{m}_\chi \mathbf{Q}(\lambda)}, \]
where $\mathbf{m}_\chi :=  \Ker \chi$ is the maximal ideal in $\overline{Z(\LL)}$ corresponding to $\chi$. By definition, $Q(\lambda,\chi)$ has radical central character $\chi$.

From \cite[Proposition 6.5]{knapp1988lie}, we have $\mathbf{Q}(\lambda) \cong \mathbf{Q}(0) \tens{\C} L(\lambda)$, and from this and the right-exactness of tensor products we conclude
\[ Q(\lambda,\chi) \cong Q(0,\chi) \tens{\C} L(\lambda). \]
In the formulas above, $L(\lambda)$ is considered to be an $\LL$-module
with the trivial action of $\rr$, and the tensor product is that of $\LL$-modules.

Since $Q(0,\chi)$ is generated by its unique copy of $L(0)$, it follows that it has the unique simple quotient $V(0,\chi)$. From \cite{mazorchuk2020lie} we know that $Q(0,\chi)=V(0,\chi)$ when $\LL$ is the Takiff $\mathfrak{sl}_2$, but generally it fails, for example when $\rr \cong L(4)$ as a $\g$-module.

\begin{proposition}[{\cite[Proposition 62]{mazorchuk2020lie}}] \label{proposition:existence_gHC}
Fix a radical central character $\chi$.
\begin{itemize}
\item $V(0,\chi)$ is the unique simple $\LL$-module having both radical central character $\chi$ and the trivial $\g$-module as one of $\g$-types.

\item If $[\g,\rr]=\rr$, then $V(0,\chi)$ is either the trivial $\LL$-module (precisely when $\chi$ is the radical central character of the trivial $\LL$-module), or infinite-dimensional.

\item $V(0,\chi)$ is a $\g$-Harish-Chandra module.
\end{itemize} 
\end{proposition}

\section{$\mathfrak{sl}_2$-Harish-Chandra modules for $\mathfrak{sl}_2 \ltimes L(4)$}

\subsection{The Lie algebra $\LL^4 = \mathfrak{sl}_2 \ltimes L(4)$}

For the remainder of this section, we fix and work with the conformal Galilei algebra $\LL := \LL^4$.

The purely radical part $\Sym(\rr)^\g$ is a polynomial algebra in two algebraically independent elements, homogeneous of degrees $2$ and $3$ (cf. \cite[Lecture XVIII]{hilbert1993theory}). This can be seen from (\ref{item:F4_2}). Moreover, one can solve a system of linear equations to get these elements explicitly:
\begin{align}
\label{al:C2}
C_2 &= v_0^2 -3 v_{-2} v_2 + 12 v_{-4} v_4, \\
\label{al:C3}
C_3 &= v_0^3 - \frac{9}{2} v_{-2} v_0 v_2 + \frac{27}{2} v_{-2}^2 v_4  + \frac{27}{2} v_{-4} v_2^2  - 36 v_{-4} v_0 v_4.
\end{align}
Therefore we identify radical central characters $\chi$ with pairs of complex numbers $(\chi(C_2), \chi(C_3))$.

\begin{proposition} \label{proposition:L4_free}
The algebra $\Sym(\rr)$ is free over the invariants $\Sym(\rr)^\g = \C[C_2,C_3]$. In particular, the $\g$-structure of the module $Q(0,\chi)$ does not depend on radical central character $\chi$.
\end{proposition}
\begin{proof}
Since neither of the elements $C_2$ and $C_3$ have the constant term, they generate a proper ideal in $\Sym(\rr)$. In other words, $C_3$ is not invertible in $\Sym(\rr)/(C_2)$.
It is not hard to see that $C_2$ is irreducible in $\Sym(\rr)$. Since $\Sym(\rr)$ is a unique factorization domain, and since $C_2$ does not divide $C_3$, we have that $C_3$ is not a zero-divisor in $\Sym(\rr)/(C_2)$.
Therefore, $(C_2,C_3)$ is a regular sequence, and the claim now follows from the main result of \cite{futorny2005kostant}.
\end{proof}


\begin{proposition} \label{proposition:Q0_mult}
The modules $Q(0,\chi)$ are $\g$-Harish-Chandra, with multiplicities:
\[ [Q(0,\chi) \colon L(l)] = \begin{cases}
\frac{l}{4}+1 & \colon l \equiv 0 \text{ (mod $4$)}, \\[5pt]
\frac{l-2}{4} & \colon l \equiv 2 \text{ (mod $4$)}, \\[5pt]
0 & \colon l \text{ odd}.\end{cases} \]
\end{proposition}
\begin{proof}
From Proposition \ref{proposition:L4_free} we know that it is enough to consider only $\chi = (0,0)$. In that case, the result follows from Proposition \ref{proposition:freeness_series} and (\ref{item:F4_3}). 
\end{proof}
\begin{remark}
\label{rem:Q00}
Note that  (\ref{item:F4_3}) gives us in fact more information than what is stated in Proposition \ref{proposition:Q0_mult}, namely, we can read off multiplicities in each degree of $Q(0,0)$ as follows:
\begin{center}
    \begin{tabular}{l|l}
         $k$  & degree $k$ part of $Q(0,0)$ \\ \hline
         $0$ & $L(0)$ \\
         $1$ & $L(4)$ \\
         $2$ & $L(4) \oplus L(8)$ \\
         $3$ & $L(6) \oplus L(8) \oplus L(12)$ \\
         $4$ & $L(8) \oplus L(10) \oplus L(12) \oplus L(16)$ \\
         $\vdots$ & $\vdots$ \\
         $k$ & $L(2k) \oplus L(2k+2) \oplus \ldots \oplus L(4k-6) \oplus L(4k-4)  \oplus L(4k)$ \\
         $\vdots$ & $\vdots$
    \end{tabular}
\end{center}
\end{remark}

\begin{remark} \label{remark:no_L2_in_Q}
From Proposition \ref{proposition:Q0_mult} we immediately get $[Q(0,\chi) \colon L(2)] =0$, which was proved in \cite{mazorchuk2020lie} in two different, more complicated ways (cf. the proof of \cite[Proposition 71]{mazorchuk2020lie}, either \cite[Lemma 68]{mazorchuk2020lie} using \cite{hahn2018from}, or Remark 69).
\end{remark}

\subsection{Highest weight modules}
Let us recall from \cite{lu2014simple} the structure of simple highest weight $\LL^4$-modules. Let $\tilde{\h} := \Span\{h,v_0\}$ denote the (generalized) Cartan subalgebra of $\LL^4$. We have the generalized triangular decomposition $\LL^4 = \tilde{\n}_- \oplus \tilde{\h} \oplus \tilde{\n}_+$, where $\tilde{\n}_- := \Span\{f,v_{-2}, v_{-4}\}$ and $\tilde{\n}_+ := \Span\{e,v_{2}, v_{4}\}$. Put also $\tilde{\bb} := \tilde{\h} \oplus \tilde{\n}_+$.

Elements  $\Lambda \in\tilde{\h}^\ast$ will be identified with pairs $(\lambda,\mu)$, where $\lambda := \Lambda(h)$ and $\mu = \Lambda(v_0)$.
Denote by $\Delta(\Lambda)$ the Verma module for $\LL^4$ with highest weight $\Lambda$, i.e.
\begin{equation*}
\Delta(\Lambda) := \Ind_{\tilde{\bb}}^\LL \C_\Lambda = U(\LL) \tens{U(\tilde{\bb})} \C_\Lambda \cong U(\tilde{\n}_-) \tens{\C} \C_\Lambda, 
\end{equation*}
where $\C_\Lambda$ is the $1$-dimensional $U(\tilde{\bb})$-module where $\tilde{\h}$ acts via $\lambda$ and $\tilde{\n}_+$ as $0$. Denote by $\mathbf{L}(\Lambda)$ the unique simple quotient of $\Delta(\Lambda)$.

We reserve the following notation for the semi-simple part $\g = \mathfrak{sl}_2$: the Verma module for $\g$ is denoted by $\Delta^{\g}(\lambda)$ and its unique simple quotient by $L(\lambda)$, cf. \cite{humphreys2008representations}.

Note that the elements $C_2$ and $C_3$ act on $\Delta(\Lambda)$ (and on its quotient $\mathbf{L}(\Lambda)$) as the scalars $\mu^2$ and $\mu^3$, respectively. It follows that there are no non-trivial homomorphisms nor extensions between Verma modules with different $\mu$'s.

From \cite[Theorem~4]{lu2014simple} we have the following:
\begin{corollary} \label{corollary:mathbbL_char}
Fix $\Lambda \in \tilde{\h}^\ast$. If $\mu=0$, then $\rr$ annihilates $\mathbf{L}(\Lambda)$, and $\mathbf{L}(\Lambda)$ is as a $\g$-module isomorphic to $L(\lambda)$. If $\mu \neq 0$, then $\mathbf{L}(\Lambda)$ has a $\g$-filtration with subquotients $\g$-Verma modules $\Delta^\g(\lambda - 2k)$, $k \in \Z_{\geq 0}$ respectively, all with multiplicity one.
\end{corollary}

Here we give more detailed description of highest weight modules, and which will later be translated to $\g$-\HC{} modules.

\begin{proposition}
\label{prop:Delta_struct}
\begin{enumerate}
\item For any $\Lambda = (\lambda,\mu)$ there is a unique up to a non-zero scalar non-zero homomorphism $\Delta(\lambda-4,\mu) \to \Delta(\lambda,\mu)$. Moreover, it is an injection.

\item If $\mu \neq 0$, then $\quotient{\Delta(\lambda,\mu)}{\Delta(\lambda-4,\mu)} \cong \mathbf{L}(\lambda,\mu)$.

\item Assume $\mu\neq 0$. The Verma module $\Delta(\lambda,\mu)$ is uniserial. More precisely, there is an infinite tower of submodules
\[ \Delta(\lambda,\mu) \supseteq \Delta(\lambda-4,\mu) \supseteq \Delta(\lambda-8,\mu) \supseteq \ldots \]
such that $\bigcap_{k \geq 0} \Delta(\lambda-4k,\mu)=0$,
with subquotients respectively $\mathbf{L}(\lambda,\mu)$, $\mathbf{L}(\lambda-4,\mu)$, $\mathbf{L}(\lambda-8,\mu)$, $\ldots$, and this tower is a unique composition series of $\Delta(\lambda,\mu)$.
\end{enumerate}
\end{proposition}
\begin{proof}
Denote by $w$ the highest weight vector in $\Delta(\lambda,\mu)$. The weight space of weight $\lambda-4$, denoted by $\Delta(\lambda,\mu)_{\lambda-4}$, is spanned by $v_{-4}\otimes w$, $v_{-2}^2\otimes w$, $fv_{-2}\otimes w$ and $f^2\otimes w$. One can check that
\[v_{-2}^2\otimes w - 4\mu v_{-4}\otimes w\]
is the only (up to scalar) $v_0$-eigenvector of eigenvalue $\mu$ annihilated by $e$, $v_2$ and $v_4$. This gives the unique non-trivial homomorphism. It is injective for the same reason as in the classical case, namely $U(\LL^4)$ is an integral domain. This proves the first claim.

The second claim is proved using Corollary \ref{corollary:mathbbL_char}, and comparing the $\g$-characters of the three modules in question. Recall that characters of Verma modules are given by Kostant's partition function. In our case, one can easily calculate
\[ \dim \Delta(\lambda,\mu)_{\lambda-2k} = \begin{cases} \frac{k^2+4k+4}{4} &\colon k \text{ even}, \\[.5 em]
\frac{k^2+4k+3}{4} &\colon k \text{ odd}. \end{cases} \]
It is enough to verify $\dim \Delta(\lambda,\mu)_{\lambda-2k} = \dim \Delta(\lambda-4,\mu)_{\lambda-2k} + \dim \mathbf{L}(\lambda,\mu)_{\lambda-2k}$ for any $k \geq 0$, which is easy and therefore omitted.

The existence part of the last claim follows directly from the first two claims. For the uniqueness, it is enough to show that there is no proper submodule $M \subseteq \Delta(\lambda,\mu)$ which is not contained in $\Delta(\lambda-4,\mu)$. But this is clear since the quotient in the second claim is simple.
\end{proof}

We will need a generalization of \cite[Lemma 72]{mazorchuk2020lie}:
\begin{proposition}
\label{prop:ext}
Assume $\mu \neq0$.
\begin{itemize}
    \item In the category of $h$-weight $\LL^4$-modules, we have:
\begin{equation*}
\Ext^1 \left( \mathbf{L}(\lambda,\mu), \mathbf{L}(\lambda',\mu') \right) = \begin{cases}
\C & \colon \mu = \mu' \text{ and } |\lambda - \lambda'| \in \{0, 4\}, \\
0 & \colon \text{otherwise}.
\end{cases} \end{equation*}
    \item In the category of $h$-weight $\LL^4$-modules with radical central character, we have:
\begin{equation*}
\Ext^1 \left( \mathbf{L}(\lambda,\mu), \mathbf{L}(\lambda',\mu') \right) = \begin{cases}
\C & \colon \mu = \mu' \text{ and } |\lambda - \lambda'| = 4, \\
0 & \colon \text{otherwise}.
\end{cases} \end{equation*}
\end{itemize}
\end{proposition}
\begin{proof}
If $\mu \neq \mu'$, then there are no non-trivial extensions because the modules have different radical central characters. So assume $\mu = \mu'$. Consider an extension
\begin{equation}
\label{eq:L_ses}
     0 \to \mathbf{L}(\lambda',\mu) \hookrightarrow M \twoheadrightarrow \mathbf{L}(\lambda,\mu) \to 0
\end{equation}
Assume that 
$\lambda' < \lambda$. From this assumption it follows that the weight space $M_\lambda$ of $M$ of weight $\lambda$ has $\dim M_\lambda =1$, so a preimage in $M$ of the highest weight vector of $\mathbf{L}(\lambda,\mu)$ is a non-zero weight vector of the same weight, annihilated by $e,v_2$ and $v_4$, and is also a $v_0$-eigenvector with eigenvalue $\mu$. Therefore, the Verma module $\Delta(\lambda,\mu)$ maps to $M$. Assume it maps subjectively, since otherwise the extension would be trivial. From Proposition \ref{prop:Delta_struct} it follows that $\lambda' = \lambda-4$, and $M$ is uniquely given as $\Delta(\lambda,\mu)/\Delta(\lambda-8,\mu)$. Note that this extension has the same radical central character.

The case $\lambda' > \lambda$ follows from applying the simple preserving duality to (\ref{eq:L_ses}), analogously as in the classical case (except that here it works only on modules of finite length).

Assume now $\lambda' = \lambda$, and that (\ref{eq:L_ses}) is not split. Then $M_\lambda$ is a two-dimensional simple $U(\tilde{\bb})$-module, on which $\tilde{\n}_+$ acts as $0$, $h$ acts diagonally as $\lambda$, and $v_0$ as $\begin{pmatrix}
\mu & 1 \\
0 & \mu \end{pmatrix}$ in some basis. Note that such a module is independent of $M$, and denote it by $\C^2_{(\lambda,\mu)}$. Denote the module $\widehat{\Delta}(\lambda,\mu) := U(\LL^4) \tens{U(\tilde{\bb})} \C^2_{(\lambda,\mu)}$, and note that it gives the unique non-split self-extension of the Verma module $0\to \Delta(\lambda,\mu) \hookrightarrow \widehat{\Delta}(\lambda,\mu) \twoheadrightarrow \Delta(\lambda,\mu) \to 0$.

By adjunction, the module $\widehat{\Delta}(\lambda,\mu)$ surjects onto $M$. Denote by $K$ the kernel of this surjection. Note that  the following diagram
\[ \xymatrix{ & 0 & 0 & 0 & \\
0 \ar[r] & \mathbf{L}(\lambda,\mu) \ar@{^{(}->}[r] \ar[u] & M \ar@{->>}[r] \ar[u]  & \mathbf{L}(\lambda,\mu) \ar[r] \ar[u]  & 0  \\
0 \ar[r] & \Delta(\lambda,\mu) \ar@{^{(}->}[r] \ar@{->>}[u] & \widehat{\Delta}(\lambda,\mu) \ar@{->>}[r] \ar@{->>}[u] & \Delta(\lambda,\mu) \ar[r] \ar@{->>}[u] & 0 \\
%
& \Delta(\lambda-4,\mu) 
\ar@{^{(}->}[u] & K 
\ar@{^{(}->}[u] & \Delta(\lambda-4,\mu)  \ar@{^{(}->}[u] 
\\
& 0 \ar[u] & 0 \ar[u] & 0 \ar[u] &   } \]
commutes, where we used Proposition \ref{prop:Delta_struct} to get the first and third column. By the $3\times 3$-lemma, we get a canonical exact sequence
\[ 0 \to \Delta(\lambda-4,\mu) \hookrightarrow K \twoheadrightarrow \Delta(\lambda-4,\mu) \to 0\]
completing the third row of the above commutative diagram. By the previous paragraph, we must have $K \cong \widehat{\Delta}(\lambda-4,\mu)$ or $K \cong \Delta(\lambda-4,\mu) \oplus \Delta(\lambda-4,\mu)$.

Denote the basis of $\C^2_{(\lambda,\mu)}$ in which $v_0$ acts as $\begin{pmatrix}
\mu & 1 \\
0 & \mu \end{pmatrix}$ by $\{w_1,w_2\}$. One can check that the subspace of $\widehat{\Delta}(\lambda,\mu)$ of weight $\lambda-4$ annihilated by $\tilde{\n}_+$ is two-dimensional, and spanned by
\[ x_1 := v_{-2}^2 \otimes w_1 -4\mu v_{-4} \otimes w_1, \qquad x_2 := -4 v_{-4} \otimes w_1 + v_{-2}^2 \otimes w_2 -4\mu v_{-4} \otimes w_2 . \]
Observe that $v_0 \cdot x_1 = \mu x_1$ and $v_0 \cdot x_2 = x_1 + \mu x_2$, which means that $\Span\{x_1,x_2\}$ is as a $\tilde{\bb}$-module isomorphic to $\C^2_{(\lambda-4,\mu)}$. This implies that $K \cong \widehat{\Delta}(\lambda-4,\mu)$, and is generated by $x_1, x_2$, which uniquely determines $M$.

Conversely,  the $\tilde{\bb}$-homomorphism $\C^2_{(\lambda-4,\mu)} \cong \Span\{x_1,x_2\} \hookrightarrow \widehat{\Delta}(\lambda,\mu)$ induces by adjunction an $\LL$-homo\-morphism $\widehat{\Delta}(\lambda-4,\mu) \to \widehat{\Delta}(\lambda,\mu)$ such that the diagram below commutes. By the $5$-lemma, this homomorphism is injective.
\[ \xymatrix{ & 0 & 0 & 0 & \\
& \mathbf{L}(\lambda,\mu) \ar[u] & \widehat{\Delta}(\lambda,\mu)/\widehat{\Delta}(\lambda-4,\mu) \ar[u]  & \mathbf{L}(\lambda,\mu) \ar[u]  & \\
0 \ar[r] & \Delta(\lambda,\mu) \ar@{^{(}->}[r] \ar@{->>}[u] & \widehat{\Delta}(\lambda,\mu) \ar@{->>}[r] \ar@{->>}[u] & \Delta(\lambda,\mu) \ar[r] \ar@{->>}[u] & 0 \\
0 \ar[r] & \Delta(\lambda-4,\mu)
\ar@{^{(}->}[u] \ar@{^{(}->}[r] & \widehat{\Delta}(\lambda-4,\mu) \ar@{->>}[r] 
\ar@{^{(}->}[u] & \Delta(\lambda-4,\mu)  \ar@{^{(}->}[u] \ar[r] & 0
\\
& 0 \ar[u] & 0 \ar[u] & 0 \ar[u] &   } \]
By the $3\times 3$-lemma, we have an exact sequence
\[ 0 \to \mathbf{L}(\lambda,\mu) \hookrightarrow \widehat{\Delta}(\lambda,\mu)/\widehat{\Delta}(\lambda-4,\mu) \twoheadrightarrow \mathbf{L}(\lambda,\mu) \to 0 , \]
which is clearly non-split. Note that $C_2, C_3$ do not act as scalars on this extension, so it does not have a radical central character.
\end{proof}

\begin{lemma}
\label{lem:Delta_split}
Assume $\mu \neq 0$.
The inclusions $\Delta(\lambda-4,\mu) \hookrightarrow \Delta(\lambda,\mu)$ (cf. Proposition \ref{prop:Delta_struct}) are split in the category of $\g$-modules.
\end{lemma}
\begin{proof}
First, observe that the claim is true if $\lambda\leq -2$, since then all $\g$-composition factors of $\Delta(\lambda,\mu)$ are anti-dominant $\g$-Verma modules (cf. Proposition \ref{prop:Delta_struct} and Corollary \ref{corollary:mathbbL_char}), and therefore cannot extend among themselves.

Next, assume that the map $\Delta(\lambda-4,\mu) \hookrightarrow \Delta(\lambda,\mu)$ is $\g$-split, for some $\lambda$. We tensor this with $L(1)$ to get a $\g$-split map
\begin{equation} \label{eq:Delta_V}
    \Delta(\lambda-5,\mu) \oplus \Delta(\lambda-3,\mu) \hookrightarrow \Delta(\lambda-1,\mu) \oplus \Delta(\lambda+1,\mu),
\end{equation}
where we use Proposition \ref{prop:ext} to see that $\Delta(\lambda,\mu) \otimes L(1)$ splits into a direct sum of Verma modules in the category of $\LL$-modules. From Proposition \ref{prop:Delta_struct} we get that (\ref{eq:Delta_V}) restricts to $\Delta(\lambda-3,\mu) \hookrightarrow  \Delta(\lambda+1,\mu)$, which is therefore also $\g$-split. The claim now follows by induction. 
\end{proof}

\subsection{Enright-Arkhipov completion}

Fix an $\ad$-nilpotent element $x \in \LL$ (for example $f$, or $v_4$, which we will use), and denote by $U_{(x)}$ the localization of the universal enveloping algebra $U:=U(\LL)$ by the multiplicative set generated by $x$. This localization satisfies the Ore conditions by \cite[Lemma 4.2.]{mathieu2000classification}. Define the $U$-$U$-bimodule $S_x := \quotient{U_{(x)}}{U}$. Recall the functor on $\LL$-modules
\[\EnAr := \prescript{e}{}{\left(S_f \tens{U} - \right)} \]
from \cite[Definition 20]{mazorchuk2020lie}. Put shortly, $\EnAr(M)$ is defined by first localizing $M$ by the multiplicative subset generated by the negative root vector $f$ in the $\mathfrak{sl}_2$-subalgebra (i.e., $U_{(f)} \tens{U} M$), then factoring the canonical image of $M$ out from the localization, and finally taking the $e$-finite part of the quotient, where $e$ is the positive root vector in the $\mathfrak{sl}_2$-subalgebra. If $f$ acts injectively on $M$ (as is the case e.g. for Verma modules), than the canonical map $M \to U_{(f)} \tens{U} M$ is injective.

The functor $\EnAr$ is a combination of two similar functors, originally given by Enright in \cite{enright1979on}, and Arkhipov in \cite{arkhipov2004algebraic}. See also \cite{deodhar1980on, andersen2003twisting, konig2002enrights, khomenko2005on}.

The functor $\EnAr$ commutes with the forgetful functor to $\g$-modules, and with tensoring with finite-dimensional modules, cf. \cite[Proposition 21]{mazorchuk2020lie}. The action of $\EnAr$ on the Verma modules for $\g$ and the indecomposable projective modules in the category $\O$ for $\g$ is given in \cite[Example 22]{mazorchuk2020lie}.

Directly from Proposition \ref{prop:Delta_struct} and Lemma \ref{lem:Delta_split} we have:

\begin{corollary}
The restriction of the functor $\EnAr$ to the category consisting of subquotients of all $\Delta(\lambda,\mu)$ with $\lambda \in \mathbb{Z}$ and $\mu \in \C \setminus\{0\}$, is exact. 
\end{corollary}

\subsection{Simple $\g$-\HC{} modules}

In \cite[Subsection~7.3]{mazorchuk2020lie} we classified simple 
$\g$-\HC{} modules for $\LL = \LL^4$ which appear in Enright-Arkhipov completions of simple highest weight modules. Using highest theory we were only able to deal with radical central characters of the form $(\mu^2,\mu^3)$ for some $\mu \in \C\setminus\{0\}$:
\begin{theorem}[{\cite[Theorem 67]{mazorchuk2020lie}}] \label{theorem:simples_in_L4}
Denote by $\F_\text{ss}$ the semi-simple additive category generated by simple subquotients of Enright-Arkhipov completions of simple highest weight $\LL^4$-modules with non-trivial radical central character. 
\begin{enumerate}
\item For each $\mu\in\mathbb{C}\setminus\{0\}$, 
there is a unique, up to isomorphism, simple object $V'(0,\mu)$ in 
$\F_\text{ss}$ on which $C_j$, where $j=2,3$, acts via $\mu^j$ 
and which has $\g$-types $L(0)$, $L(4)$, $L(8),\dots$,
all multiplicity free. 
\item For each $\mu\in\mathbb{C}\setminus\{0\}$, 
there is a unique, up to isomorphism, simple object $V'(2,\mu)$ in 
$\F_\text{ss}$ on which $C_j$, where $j=2,3$, acts via $\mu^j$ 
and which has $\g$-types $L(2)$, $L(6)$, $L(10),\dots$,
all multiplicity free.
\item For each $\mu \in\mathbb{C}\setminus\{0\}$
and for each $i\in\Z_{>0}$, there is a unique, up to isomorphism,
simple object $V(i,\mu)$ in $\F_\text{ss}$ on which $C_j$, where $j=2,3$,
acts via $\mu^j$ and which has $\g$-types $L(i)$, $L(i+2)$, $L(i+4),\dots$,
all multiplicity free.
\item The modules above provide a complete and irredundant
list of representatives of isomorphism classes of simple objects in $\F_\text{ss}$.

\item Let $V$ be a simple $\g$-Harish-Chandra module on which $C_j$, where $j=2,3$, acts via $\mu^j$, for some $\mu\in\mathbb{C}\setminus\{0\}$. Then $V$ belongs to $\F_\text{ss}$.
\end{enumerate}
\end{theorem}

Note that for $\chi = (\mu^2,\mu^3)$, $\mu \in \C \setminus \{0\}$, the module $V(0,\chi)$ (defined in Section \ref{section:existente_gHC}) is denoted by $V'(0,\mu)$ in Theorem \ref{theorem:simples_in_L4}.

A simple $\g$-Harish-Chandra module with trivial radical central character is isomorphic to a finite-dimensional simple $\g$-modules with the trivial action of $\rr$. This follows from Proposition \ref{proposition:existence_gHC} and the self-adjunction of tensoring with a finite-dimensional simple $\g$-module.

As a construction part of the proof of Theorem \ref{theorem:simples_in_L4}, we had proved:
\begin{proposition}[{\cite[Subsection~7.3]{mazorchuk2020lie}}]
\label{prop:EAL}
Fix $\Lambda$ with $\lambda \in \Z$ and $\mu \neq 0$. Then:
\[ \EnAr(\mathbf{L}(\Lambda)) \cong \begin{cases} V'(0,\mu) \oplus V'(2,\mu) & \colon \lambda = -2, \\
V(|\lambda+2|,\mu) & \colon \lambda \neq -2. \end{cases} \]
\end{proposition}

\subsection{Projective $\g$-\HC{} modules}

Now we consider the category $\F$ of finitely generated $\g$-Harish-Chandra modules for $\LL^4$ with fixed radical central character $\chi=(\mu^2,\mu^3)$, for $\mu \in \C\setminus\{0\}$. Note that $\F$ contains no finite-dimensional modules.
This is an abelian category with finite-dimensional $\Hom$-spaces. By \cite[Lemma 5.2 and Theorem 5.5]{krause2015krull}, $\F$ is a Krull-Schmidt category. Note that $\F$ is Noetherian, but not Artinian.

Theorem \ref{theorem:simples_in_L4} describes simple modules in $\F$. When there is no possibility of confusion, we write the simple modules as $V'(0)$, $V'(2)$, $V(1)$, $V(2)$, $V(3)$, $\ldots$, as the radical central character is assumed to be fixed. Moreover, we will write $Q(k)$ instead of $Q(k,\chi)$.

\begin{lemma}
\label{lem:LV}
The following formulae hold:
\begin{align*}
&L(k) \otimes V'(0) \cong \begin{cases}
    V'(0) \oplus V(2) \oplus V(4) \oplus \ldots \oplus V(k)  &\colon k \equiv 0 \text{ (mod $4$)}, \\[.5em]
    V'(2) \oplus V(2) \oplus V(4) \oplus \ldots \oplus V(k)  &\colon k \equiv 2 \text{ (mod $4$)}, \\[.5em]
    V(1) \oplus V(3) \oplus V(5) \oplus\ldots  \oplus V(k)  &\colon k \equiv 1 \text{ (mod $2$)},
    \end{cases} \\[.5em]  
&L(k) \otimes V'(2) \cong \begin{cases}
    V'(2) \oplus V(2) \oplus V(4) \oplus \ldots \oplus V(k)  &\colon k \equiv 0 \text{ (mod $4$)}, \\[.5em]
    V'(0) \oplus V(2) \oplus V(4) \oplus \ldots \oplus V(k)  &\colon k \equiv 2 \text{ (mod $4$)}, \\[.5em]
    V(1) \oplus V(3) \oplus V(5) \oplus\ldots  \oplus V(k)  &\colon k \equiv 1 \text{ (mod $2$)},
    \end{cases} \\[.5em]
&L(k) \otimes V(n) \cong \begin{cases}
    V(n-k) \oplus V(n-k+2) \oplus \ldots \oplus V(n+k-2) \oplus V(n+k)  &\colon k<n, \\[.5em]
    V'(0) \oplus V'(2) \oplus V(2) \oplus V(4) \oplus \ldots \oplus V(2n-2) \oplus V(2n)  &\colon k=n, \\[.5em]
    V(1)^{\oplus 2} \oplus V(3)^{\oplus 2} \oplus V(5)^{\oplus 2} \oplus \ldots \oplus V(k-n)^{\oplus 2} \oplus{}  \\  \qquad \qquad {}\oplus V(k-n+2) \oplus V(k-n+4) \oplus \ldots \oplus V(k+n) &\colon k>n, \ k \not\equiv n \text{ (mod $2$)}, \\[.5em]
    V'(0) \oplus V'(2) \oplus V(2)^{\oplus 2} \oplus V(4)^{\oplus 2} \oplus \ldots \oplus V(k-n)^{\oplus 2} \oplus{}  \\  \qquad \qquad {}\oplus V(k-n+2) \oplus V(k-n+4) \oplus \ldots \oplus V(k+n) &\colon k>n, \ k \equiv n \text{ (mod $2$)}.
    \end{cases}
\end{align*}
\end{lemma}
\begin{proof}
Follows from \cite[Remark 74]{mazorchuk2020lie} and induction on $k$.
\end{proof}

Note that $P'(0) := Q(0) \twoheadrightarrow V'(0)$ is a projective cover of $V'(0)$ in $\F$, which follows from the construction of $Q(0)$. Moreover, since $L(k) \otimes -$ is exact and preserves projective objects, from Lemma \ref{lem:LV} we see that any simple object in $\F$ is a quotient of some projective module $Q(k)$. Therefore, each simple object in $\F$ has a unique up to isomorphism indecomposable projective cover in $\F$, and these exhaust all indecomposable projective objects in $\F$. We denote them by $P'(0)$, $P'(2)$, $P(1)$, $P(2)$, $P(3)$, $\ldots$ in the obvious notational relation with the simple objects.

\begin{proposition}
\label{prop:Qk}
Let $P\twoheadrightarrow V$ be a projective cover of a simple module in $\F$. Then the multiplicity of $P$ in the decomposition of $Q(k)$ into indecomposables is exactly $[V \colon L(k)]$. In particular,
\begin{equation}
\label{eq:Q_P}
    Q(k) = \begin{cases}
    P'(0) \oplus P(2) \oplus P(4) \oplus \ldots \oplus P(k)  &\colon k \equiv 0 \text{ (mod $4$)}, \\[.5em] 
    P'(2) \oplus P(2) \oplus P(4) \oplus \ldots \oplus P(k)  &\colon k \equiv 2 \text{ (mod $4$)}, \\[.5em] 
    P(1) \oplus P(3) \oplus P(5) \oplus \ldots \oplus P(k)  &\colon k \equiv 1 \text{ (mod $2$)}. 
    \end{cases}
\end{equation}
\end{proposition}
\begin{proof}
The multiplicity of $P$ in $Q(k)$ is equal to $\dim \Hom_{\LL}(Q(k),V) = \dim \Hom_{\g}(L(k),V) = [V \colon L(k)]$. The rest follows now from Theorem \ref{theorem:simples_in_L4}.
\end{proof}

\begin{proposition}
\label{prop:P0_subq}
The module $P'(0)$ has a filtration with simple subquotients $V'(0), V(4), V(8), V(12), \dots$.
\end{proposition}
\begin{proof}
Follows by comparing the $\g$-multiplicities of $Q(0)$ according to Proposition \ref{proposition:Q0_mult} (cf. also Remark \ref{rem:Q00}) and of the simple modules from Theorem \ref{theorem:simples_in_L4}. 
\end{proof}

\begin{proposition}
\label{prop:P_odd}
The module $P(2k+1)$, $k \geq 0$, has a filtration with simple subquotients $V(1), V(3), V(5), \dots$.
\end{proposition}
\begin{proof}
We use induction on $k$. For $k=0$ we have $P(1)=Q(1)=L(1)\otimes Q(0)$, and the claim follows from Proposition \ref{prop:P0_subq} and Lemma \ref{lem:LV}. For $k=1$, observe that for a simple $\g$-\HC{} module $V$ we have $\Hom_{\LL}(L(2) \otimes P(1), V) \cong \Hom_{\LL}(P(1), L(2) \otimes V)$. By Lemma \ref{lem:LV}, the latter $\Hom$ is non-zero only if $V=V(1)$ or $V=V(3)$, and having dimension $2$ and $1$, respectively. It follows that
\begin{equation*}
    L(2) \otimes P(1) \cong P(1) \oplus P(1) \oplus P(3).
\end{equation*}
We can get $\g$-types of the left-hand side above using the case  $k=0$, and Lemma \ref{lem:LV}, from which the claim of the proposition follows for $k=1$. Similarly, one can check that for $k>1$ we have
\begin{equation*}
    L(2) \otimes P(2k+1) \cong P(2k-1) \oplus P(2k+1) \oplus P(2k+3),
\end{equation*}
from which the general statement follows by induction.
\end{proof}

\begin{proposition} 
\label{prop:P_comps}
\begin{itemize}
    \item $P'(2)$ has a filtration with simple subquotients $V'(2)$, $V(4)$, $V(8)$,  $V(12)$, $\dots$.
    \item  For $k\geq 0$, the module $P(4k+2)$ has a filtration with simple subquotients $V(2)$, $V(2)$,  $V(6)$, $V(6)$, $V(10)$, $V(10)$, $\dots$.
    \item  For $k\geq 1$, the module $P(4k)$ has a filtration with simple subquotients $V'(0)$, $V'(2)$, $V(4)$, $V(4)$,  $V(8)$, $V(8)$, $\dots$.
\end{itemize}
\end{proposition}
\begin{proof}
Note that $L(1) \otimes P'(2) \cong P(1)$, again by self-adjunction of $L(1) \otimes -$ and Lemma \ref{lem:LV}. The composition factor $V'(2)$ must appear in $P'(2)$, and the rest of composition factors of $P'(2)$ are uniquely determined, and can be easily reconstructed from the above isomorphism, by Lemma \ref{lem:LV} and Proposition \ref{prop:P_odd}. This proves the first part.

Similarly, we have $L(1) \otimes P(1) \cong P'(0) \oplus P'(2) \oplus P(2)$, which reveals the $\g$-types of $P(2)$ to be $V(2)$, $V(6)$, $V(10)$, $\ldots$, each with multiplicity $2$. Finally, for $k\geq 1$ we have
\begin{equation*}
    L(1) \otimes P(2k+1) \cong P(2k) \oplus P(2k+2),
\end{equation*}
from which, together with Proposition \ref{prop:P_odd}, the last two claims follow by induction.
\end{proof}

\subsection{Extensions and the quiver of $\F$}

Denote
\[ \mathbf{N} := \Ind_\g^\LL \Delta^\g(-2), \qquad N := \quotient{\mathbf{N}}{\mathbf{m}_\chi \mathbf{N}}, \]
where $\mathbf{m}_\chi :=  \Ker \chi$ is the maximal ideal in $\overline{Z(\LL)}$ corresponding 
to our fixed radical central character $\chi = (\mu^2,\mu^3)$, for $\mu \in \C\setminus\{0\}$. 
Note that, directly from the definitions, we have the canonical surjection
\begin{equation}
\label{eq:varphi}
   \varphi \colon N \twoheadrightarrow \Delta(-2,\mu).
\end{equation}

\begin{lemma}
\label{lem:ker}
The kernel $\Ker \varphi$ is isomorphic, as a $\g$-module, to a direct sum of 
projective modules in category $\O$ for $\g$.
\end{lemma}
\begin{proof}
By the PBW Theorem, as a vector space, the induction from $\g$ to $\LL$ is given by tensoring
with  $U(\rr)$ over $\mathbb{C}$. Each homogeneous component of 
$U(\rr)$ is finite dimensional and stable under the adjoint action of $\g$. 
Therefore we may consider the part of the induction given by tensoring with 
each such homogeneous component separately. Reducing further modulo the 
radical central character, the decomposition of this module can be read 
off from Remark \ref{rem:Q00}. In particular, the homogeneous degree 
$k$ part of $N$ is given by
\begin{equation}\label{eqn1}
\left(L(2k) \oplus L(2k+2) \oplus \ldots \oplus L(4k-6) \oplus L(4k-4)  \oplus L(4k)\right)\otimes \Delta^\g(-2). 
\end{equation}

For $m\geq 0$, using \cite[Subsection~5.8]{mazorchuk2010lectures}, we have 
\begin{equation}\label{eqn2}
L(2m)\otimes \Delta^\g(-2)\cong
\Delta^\g(-2-2m)\oplus P^\g(-2-2m+2)\oplus P^\g(-2-2m+4)\oplus \dots \oplus P^\g(-2).
\end{equation}
Modulo the projective part, the highest weight of the summand $\Delta^\g(-2-2m)$
is given by tensoring the highest weight of $\Delta^\g(-2)$ with the lowest weight
of $L(2m)$.

In the Verma module $\Delta(-2,\mu)$, the $\g$-summand corresponding to $k$ is a
direct sum of the $\Delta^\g(-2-2m)$'s, for $m\in\{k,k+1,\dots,2k\}$. Comparing
now \eqref{eqn1} and \eqref{eqn2}, we see that all we need to show is that, 
for each summand $L(2k+2i)$ in \eqref{eqn1}, the corresponding summand 
$\Delta^\g(-2-2(k+i))$ in \eqref{eqn2} does not belong  to $\Ker \varphi$.
Since $\Delta^\g(-2-2(k+i))$ comes from the lowest weight of $L(2k+2i)$, it is enough
to show that the elements of the form $v_{-4}^iv_{-2}^{k-i}$, where $i=0,1,\dots,k$,
which obviously survive the projection $\varphi$
(together with all their linear combinations), generate all these lowest weights.
In other words, we need to show that the adjoint $\g$-submodule of $U(\rr)$ generated by
all elements of the form $v_{-4}^iv_{-2}^{k-i}$, where $i=0,1,\dots,k$,
is isomorphic to the $\g$-module given by the decomposition \eqref{eqn1}.
This is a non-trivial combinatorial statement which will be proved separately
in Section~\ref{s-new} independently of the rest of the paper.
This completes the proof.
\end{proof}

\begin{lemma}
\label{lem:QinEA}
There is an inclusion
\[ Q(0) \hookrightarrow \EnAr (\Delta(-2,\mu)). \]
\end{lemma}
\begin{proof}
Let us apply $\EnAr$ to (\ref{eq:varphi}). Note that the modules there have filtrations with subquotients isomorphic to Verma modules for $\g$. The functor $\EnAr$ is left exact on modules with $\g$-Verma 
flags since $S_f\otimes_{U}{}_-$ is exact on such modules by \cite[Theorem~2.2]{andersen2003twisting}
and taking the $e$-locally finite submodule is manifestly left exact. 
Therefore, we get the exact sequence
\[ 0 \to \EnAr(\Ker \varphi) \hookrightarrow \EnAr(N) \to  \EnAr(\Delta(-2,\mu)) . \]

Because of Lemma \ref{lem:ker} and \cite[Example 22]{mazorchuk2020lie} we have $\EnAr(\Ker \varphi) = 0$. Moreover, since $\EnAr$ commutes with tensoring with finite-dimensional modules (cf. \cite[Proposition 21]{mazorchuk2020lie}), we have that it also commutes with $\Ind_\g^\LL$, and hence $\EnAr(\mathbf{N}) \cong \mathbf{Q}(0)$. Since $\EnAr$ commutes with quotienting to the correct radical central character, we have $\EnAr(N) \cong Q(0)$.
\end{proof}

\begin{proposition}
\label{prop:radQ0}
The modules $P'(0) = Q(0)$ and $P'(2)$ have the radical filtrations respectively as follows: \[\xymatrix@R=1em{V'(0) , \ar@{-}[d] \\ V(4) \ar@{-}[d] \\ V(8) \ar@{-}[d] \\ \vdots} \qquad\qquad\qquad\qquad
\xymatrix@R=1em{V'(2) . \ar@{-}[d] \\ V(4) \ar@{-}[d] \\ V(8) \ar@{-}[d] \\ \vdots} \]
\end{proposition}
\begin{proof}
From Lemma \ref{lem:QinEA}, Proposition \ref{prop:Delta_struct} and Proposition \ref{prop:EAL} we see that the only composition factors that can extend in $Q(0)$ have parameters that differ by $4$. Since $Q(0)$ is indecomposable, the claim about its structure follows. Taking this into account, and since by Proposition  \ref{prop:Qk} we have $P'(0) \otimes L(2) \cong P'(2) \oplus P(2)$, we can see that the only way the composition factors of $P'(2)$ (cf. Proposition \ref{prop:P_comps}) can be arranged is exactly as claimed.
\end{proof}

\begin{proposition}
\label{prop:loewy}
The first two layers of the radical filtrations of the projective modules $P(k)$, $1 \leq k\leq 4$, are given as
\[
 \xymatrix@R=1em@C=0em{ & V(1) \ar@{-}[ld] \ar@{-}[rd] \\ V(3) && V(5) ,  } \qquad 
\xymatrix@R=1em@C=0em{ & V(2) \ar@{-}[ld] \ar@{-}[rd] \\ V(2) && V(6) ,  }  \qquad 
\xymatrix@R=1em@C=0em{ & V(3) \ar@{-}[ld] \ar@{-}[rd] \\ V(1) && V(7) ,  }  \qquad 
\xymatrix@R=1em@C=.5em{ & V(4) \ar@{-}[ld] \ar@{-}[rd] \ar@{-}[d] \\ V'(0) & V'(2) & V(8) ,  }
\]
and for $k\geq 5$ as
\[
\xymatrix@R=1em@C=0em{ & V(k) \ar@{-}[ld] \ar@{-}[rd] \\ V(k-4) && V(k+4) .  }
\]
\end{proposition}
\begin{proof}
It follow from tensoring the radical filtration of $P'(0)$ from Proposition \ref{prop:radQ0} with $L(k)$ using Lemma \ref{lem:LV}, and mathematical induction for $k \geq 5$.
\end{proof}

As a corollary we now have a more precise version of Proposition \ref{prop:P_comps}:

\begin{corollary}
The radical filtrations of the projective modules $P(k)$, $k \geq 1$, are given as follows:
\[
\xymatrix@R=1em@C=0em{ & V(k) \ar@{-}[ld] \ar@{-}[rd] \\ V(k-4) \ar@{-}[d] && V(k+4) \ar@{-}[d] \\
V(k-8) \ar@{-}[d] && V(k+8) \ar@{-}[d] \\
\vdots && \vdots  }
\]
where the left branch eventually runs into one of the following four options, depending on $k \text{ (mod $4$)}$ respectively:
\[ \xymatrix@R=1em@C=0em{ & \vdots \ar@{-}[d] \\ & V(4) \ar@{-}[ld] \ar@{-}[rd] \\ V'(0) \ar@{-}[dr] && V'(2) \ar@{-}[dl] \\ & V(4) \ar@{-}[d] \\ & V(8) \ar@{-}[d] \\ & \vdots  } \qquad\qquad
\xymatrix@R=1em@C=0em{ \vdots \ar@{-}[d] \\ V(1) \ar@{-}[d] \\ V(3) \ar@{-}[d] \\ V(7) \ar@{-}[d] \\ V(11) \ar@{-}[d] \\ \vdots }   \qquad\qquad 
\xymatrix@R=1em@C=0em{ \vdots \ar@{-}[d] \\ V(2) \ar@{-}[d] \\ V(2) \ar@{-}[d] \\ V(6) \ar@{-}[d] \\ V(10) \ar@{-}[d] \\ \vdots }   \qquad\qquad 
\xymatrix@R=1em@C=0em{ \vdots \ar@{-}[d] \\ V(3) \ar@{-}[d] \\ V(1) \ar@{-}[d] \\ V(5) \ar@{-}[d] \\ V(9) \ar@{-}[d] \\ \vdots }     \]
\end{corollary}
\begin{proof}
Note that Proposition \ref{prop:loewy} gives all the extensions between simple objects in $\F$. The claim follows by induction on $k$, using the fact that $L(1) \otimes P'(0)  \cong P(1)$, $L(1) \otimes P(1)  \cong P'(0) \oplus P'(2) \oplus P(2)$, and $L(1) \otimes P(k)  \cong P(k-1) \oplus P(k+1)$ for $k \geq 1$.
\end{proof}

The main result of the paper is the following theorem.

\begin{theorem}
\label{thm:quiver}
The category $\F$ decomposes into three blocks, with the following Gabriel quivers:
\begin{enumerate}
    \item The quiver
    \[\begin{tikzcd}
    \ldots\phantom{V} \arrow[r, shift left, bend left=10] & V(11)  \arrow[l, shift left, bend left=10] \arrow[r, shift left, bend left=10]  & V(7) \arrow[l, shift left, bend left=10] \arrow[r, shift left, bend left=10]  & V(3) \arrow[l, shift left, bend left=10] \arrow[r, shift left, bend left=10]  & V(1) \arrow[l, shift left, bend left=10] \arrow[r, shift left, bend left=10]  & V(5) \arrow[l, shift left, bend left=10] \arrow[r, shift left, bend left=10]   & V(9) \arrow[l, shift left, bend left=10] \arrow[r, shift left, bend left=10]  & \phantom{V}\ldots \arrow[l, shift left, bend left=10]
    \end{tikzcd}\]
    with relations that all $2$-cycles are equal to zero;
    
    \item The quiver
    \[\begin{tikzcd}
    & V(2) \arrow[loop left] \arrow[r, shift left, bend left=10]  & V(6) \arrow[l, shift left, bend left=10] \arrow[r, shift left, bend left=10]  & V(10) \arrow[l, shift left, bend left=10] \arrow[r, shift left, bend left=10]  & V(14) \arrow[l, shift left, bend left=10] \arrow[r, shift left, bend left=10]  &  \phantom{V}\ldots \arrow[l, shift left, bend left=10]
    \end{tikzcd} \]
    
    with the following relations:
    \begin{enumerate}[(i)]
        \item all $2$-cycles are equal to zero,
        \item the loop at $V(2)$ squares to zero;
    \end{enumerate}
    
    \item The quiver \[\begin{tikzcd}
    V'(0) \arrow[rd, shift left, bend left=10, "a"] \\
    & V(4) \arrow[r, shift left, bend left=10]  \arrow[lu, shift left, bend left=10, "b"] \arrow[ld, shift left, bend left=10, "d"] & V(8) \arrow[l, shift left, bend left=10] \arrow[r, shift left, bend left=10]  & V(12) \arrow[l, shift left, bend left=10] \arrow[r, shift left, bend left=10] &  \phantom{V}\ldots \arrow[l, shift left, bend left=10] \\
    V'(2)  \arrow[ru, shift left, bend left=10, "c"]
    \end{tikzcd} \]
    with the following relations:
    \begin{enumerate}[(i)]
        \item all $2$-cycles, except $ab$ and $cd$, are zero,
        \item $ab=cd$,
        \item $da=0$, $bc=0$.
    \end{enumerate}
\end{enumerate}
\end{theorem}
\begin{proof}
From the radical filtrations of the indecomposable projectives obtained above
it follows that all relations described in the formulation of the
theorem must hold. On the other hand, it is easy to check that the algebras
described by these relations have exactly the right radical filtrations
of the indecomposable projective modules. The claim follows.
\end{proof}

\subsection{On $\g$-\HC{} modules with generic radical central character}
\label{sub:L4_generic}

Because of the form of the central elements (\ref{al:C2}) and (\ref{al:C3}), one cannot obtain $\g$-\HC{} modules with radical central character $\chi$ different from $(\mu^2,\mu^3)$ for some $\mu$ by localizing highest weight modules. We suspect that, in this case, the module $Q(0,\chi)$ is simple. To develop our theory further to cover these cases, one would need a generalization of the highest weight theory for this Lie algebra.

\begin{conjecture}
If $\chi \neq (\mu^2,\mu^3)$, for all $\mu \in \C$, then the module $Q(0,\chi)$ is simple.
\end{conjecture}

\section{Difficulties with other conformal Galilei algebras}

\subsection{The Lie algebra $\LL^3 = \mathfrak{sl}_2 \ltimes L(3)$}

The radical part of center $\Sym(\rr)^\g$ for $\LL^3$ is a polynomial algebra in one generator, the so called cubic discriminant
\[ C_4 := v_{-1}^2 v_{1}^2 - 27 v_{-3}^2 v_{3}^2 - 4 v_{-1}^3v_3 - 4 v_{-1}v_3^3 + 18v_{-3}v_{-1}v_1 v_3 , \]
cf. \cite[Lecture XVIII]{hilbert1993theory} (the proof follows from  (\ref{item:F3_2}), up to solving a system of linear equations to determine exactly the coefficients of $C_4$). So, radical central characters are identified with complex numbers. From the main result of \cite{futorny2005kostant} it follows directly that the algebra $\Sym(\rr)$ is free as a module the invariants $\Sym(\rr)^\g = \C[C_4]$.

From Theorem \ref{proposition:existence_gHC} we have:
\begin{corollary}\label{corollary:V0_HC}
For $\chi \neq 0$ the module $V(0,\chi)$ is an infinite-dimensional simple $\g$-Harish-Chandra module.
\end{corollary}

\begin{remark}
Corollary \ref{corollary:V0_HC} also follows from the fact that $Q(0,\chi)$ is a $\g$-Harish-Chandra module, cf. \cite[Proposition 64]{mazorchuk2020lie}, which is proved using results from \cite{hahn2018from}. A more direct proof of the fact that $Q(0,\chi)$ is a $\g$-Harish-Chandra module, and also of the $\g$-multiplicity formula
\[[Q(0,\chi) \colon L(l)] = l- 2\left\lfloor\frac{l+2}{3} \right\rfloor+1 , \]
follows directly from Proposition \ref{proposition:freeness_series} and (\ref{item:F3_3}). This, together with Remark \ref{remark:no_L2_in_Q}, makes the results in the paper \cite{mazorchuk2020lie} independent of \cite{hahn2018from}.
\end{remark}

\begin{conjecture}
For $\chi \neq 0$ the module $Q(0,\chi)$ is simple.
\end{conjecture}

Unfortunately one cannot apply highest-weight theory to study $\g$-Harish-Chandra modules for $\LL^3$, since there is no zero-weight space in $L(3)$. Therefore, all modules one might obtain this way are of trivial radical central character, and therefore uninteresting. As in Subsection \ref{sub:L4_generic}, a hope is to find a correct generalized notion of highest weight modules, that would localize to $\g$-Harish-Chandra modules for with a non-trivial radical central character.

\subsection{The Lie algebra $\LL^5 = \mathfrak{sl}_2 \ltimes L(5)$}
\label{sub:L5}

The purely radical part of the center of $U(\LL^5)$ is not a polynomial algebra. It is generated by four homogeneous elements, with degrees $4$, $8$, $12$, $18$, and there is one algebraic relation among them of order $36$. The proof follows from (\ref{item:F5_2}). This was known already in \cite[Lecture XVIII]{hilbert1993theory}).

Moreover, $\Sym \rr$ is not free over the invariant algebra $\left(\Sym \rr \right)^\g$. This follows from Proposition \ref{proposition:freeness_series} and (\ref{item:F5_3}). Therefore, the $\g$-structure of the module $Q(0,\chi)$ may depend on $\chi$.

\subsection{The Lie algebra $\LL^6 = \mathfrak{sl}_2 \ltimes L(6)$}
\label{sub:L6}

Similarly to Subsection \ref{sub:L5}, the purely radical part of the center of $U(\LL^6)$ is not a polynomial algebra. It is generated by five homogeneous elements, with degrees $2$, $4$, $6$, $10$, $15$, and there is one algebraic relation among them of order $30$. The proof follows from  (\ref{item:F6_2}) (and was presumably known from the classical invariant theory).

Moreover, $\Sym \rr$ is not free over the invariant algebra $\left(\Sym \rr \right)^\g$. This follows from Proposition \ref{proposition:freeness_series} and (\ref{item:F6_3}). Therefore, the $\g$-structure of the module $Q(0,\chi)$ may depend on $\chi$.

\section{Proof of Lemma~\ref{lem:ker}}\label{s-new}

\subsection{The Young lattice}\label{s-new.1}

Consider the simple oriented graph $\Lambda$ representing the Hasse diagram of the Young lattice.
The vertices of $\Lambda$ are all partitions $\lambda=(\lambda_1,\lambda_2,\dots,\lambda_k)$,
where $\lambda_1\geq \lambda_2\geq \dots$,
of all non-negative integers, that is $\lambda\vdash n$, for $n\in\mathbb{Z}_{\geq 0}$.
As usual, we identify partitions with the corresponding Young diagrams. We refer to 
\cite{sagan} for the standard terminology related to Young diagrams and partitions. 
For two vertices, $\lambda$ and $\mu$, there is a directed edge from $\lambda$ to $\mu$
if and only if $\lambda$ can be obtained from $\mu$ be removing one removable node.
For $n\in\mathbb{Z}_{\geq 0}$, let $\Lambda_n$ denote the set of all vertices 
of $\Lambda$ that are partitions of $n$.

We will be interested in the full subgraph $\Lambda^{(4)}$ of $\Lambda$ whose vertices are
all partitions $\lambda$ such that $\lambda_1\leq 4$. We fix a variable $x$ and label the 
edges of $\Lambda^{(4)}$ as follows:
\begin{itemize}
\item if $\lambda$ is obtained from $\mu$ by removing a node in the first column of $\mu$,
then the edge from $\lambda$ to $\mu$ is labeled by $x-c$, where $c$ is the length of the
first column in $\lambda$;
\item if $\lambda$ is obtained from $\mu$ by removing a node in the $m$'s column of $\mu$,
where $m>1$, then the edge from $\lambda$ to $\mu$ is la belled by the number of parts of $\lambda$
that are equal to $m-1$.
\end{itemize}
In Figure~\ref{fig-n1} one finds a depiction of an initial part of $\Lambda^{(4)}$.

\begin{figure}
\resizebox{\textwidth}{!}{
$
\xymatrix{
&&&&&&&&&&\text{\yng(4)}\ar[rr]^{x-1}&&\text{\yng(4,1)}&&\\
&&&&&&\text{\yng(2)}\ar[rr]^{1}\ar[rrd]^{x-1}&&\text{\yng(3)}\ar[rru]^{1}\ar[rr]^{x-1}
&&\text{\yng(3,1)}\ar[rru]^{1}\ar[rr]^{1}\ar[rrd]^>>>>>>>{x-2}&&
\text{\yng(3,2)}&&\\
&&\varnothing\ar[rr]^{x}&&\text{\yng(1)}\ar[rrd]^{x-1}\ar[rru]^{1}
&&&&\text{\yng(2,1)}\ar[rru]^{1}\ar[rr]^{1}\ar[rrd]^{x-2}
&&\text{\yng(2,2)}\ar[rru]^>>>>>>>>>>{2}\ar[rrd]^>>>>>>>>{x-2}&&\text{\yng(3,1,1)}&&\\
&&&&&&\text{\yng(1,1)}\ar[rru]^{2}\ar[rr]^{x-2}&&
\text{\yng(1,1,1)}\ar[rr]^{3}\ar[rrd]^{x-3}
&&\text{\yng(2,1,1)}\ar[rru]^>>>>>>>>{1}\ar[rr]^{2}\ar[rrd]^{x-3}
&&\text{\yng(2,2,1)}&&\\
&&&&&&&&&&\text{\yng(1,1,1,1)}\ar[rr]^{4}\ar[rrd]^{x-4}&&\text{\yng(2,1,1,1)}&&\\
&&&&&&&&&&&&\text{\yng(1,1,1,1,1)}&&\\
}
$
}
\caption{Initial part of $\Lambda^{(4)}$}\label{fig-n1}
\end{figure}

\subsection{The matrix}\label{s-new.2}

For an oriented path $p$ in $\Lambda^{(4)}$, denote by $\gamma_p$ the product of
the labels of all edges that form $p$. We set $\gamma_p=1$, if $p$ is the trivial path
for some vertex.

For $n\geq 1$, denote by $M_n$ the matrix
\begin{itemize}
\item whose rows are indexes by the partitions $(1^k)$, for $1\leq k\leq n$,
\item whose columns are indexes by the elements in $\Lambda^{(4)}_n$,
\item the element in the intersection of the row indexed by $(1^k)$ and
the column indexed by $\lambda$ is the sum of $\gamma_p$'s taken over all
oriented path from $(1^k)$ to $\lambda$ in $\Lambda^{(4)}$ (as usual, the empty sum equals $0$).
\end{itemize}
Here are some examples of $M_n$, for small $n$ (here the first column shows the indexes for
all rows and the first row shows the  indexes for all columns):
\begin{gather*}
M_1:
\begin{array}{c||c}
&(1)\\
\hline\hline
(1)&1
\end{array}
\qquad\qquad
M_2:
\begin{array}{c||c|c}
&(2)&(1^2)\\
\hline\hline
(1)&1&x-1\\
\hline
(1^2)&0&1
\end{array}
\qquad\qquad
M_3:
\begin{array}{c||c|c|c}
&(3)&(2,1)&(1^3)\\
\hline\hline
(1)&1&3(x-1)&(x-1)(x-2)\\
\hline
(1^2)&0&2&x-2\\
\hline
(1^3)&0&0&1
\end{array}\\
M_4:
\begin{array}{c||c|c|c|c|c}
&(4)&(3,1)&(2^2)&(2,1^2)&(1^4)\\
\hline\hline
(1)&1&4(x-1)&3(x-1)&6(x-1)(x-2)&(x-1)(x-2)(x-3)\\
\hline
(1^2)&0&2&2&5(x-2)&(x-2)(x-3)\\
\hline
(1^3)&0&0&0&3&x-3\\
\hline
(1^4)&0&0&0&0&1
\end{array}
\end{gather*}
We note that $M_n$ is not a square matrix, in general.
Also, for $n>4$, it will always contain rows in which all
non-zero entries have some factor of the form $x-l$.
Our main observation about $M_n$ is the following:

\begin{proposition}\label{prop-nnn1}
After the evaluation $x=n$, the matrix $M_n$ has rank $n$. 
\end{proposition}

\begin{proof}
We prove this claim using induction on $n$, where the evaluation of $x$ to $n$
is made at the very last step. The strategy of the proof  is as follows: if we delete the
row and the column corresponding to $(1^n)$ from $M_n$, all columns of the 
resulting matrix $\widetilde{M}_n$ are, by construction, linear combinations of 
the columns in $M_{n-1}$. 

Note that, going from $M_{n-1}$ to $M_{n}$, the rank is supposed to increase by $1$.
This increase by $1$ corresponds to the row/column of $(1^n)$ which we just deleted.
So, we only need to show that the linear span of the columns of $\widetilde{M}_n$
coincides with the linear span of the columns of $M_{n-1}$.

We will construct an injection  $\psi:\Lambda^{(4)}_{n-1}\to\Lambda^{(4)}_{n}\setminus\{(1^n)\}$ 
and consider the matrix $N_n$ 
\begin{itemize}
\item whose rows are indexed by $\psi(\Lambda^{(4)}_{n-1})$,
\item whose columns are indexed by $\Lambda^{(4)}_{n-1}$,
\item the entry in the intersection of the row $\psi(\lambda)$ and the column $\mu$
equals the label of the arrow from $\mu$ to $\psi(\lambda)$
(treated as zero if $\mu\not\to \psi(\lambda)$).
\end{itemize}
Note that $N_n$ is a square matrix by construction.
We will prove that the determinant of $N_n$ is non-zero.
Moreover, we will show that, up to a non-zero integer factor, this determinant
is a product of factors of the form $x-i$, where $i<n$. In particular,
the determinant for each $N_m$, where $m\leq n$, remains non-zero after the evaluation $x=n$.

Unfortunately, the construction is not uniform and depends on the parity of $n$.
The case of even $n$ is easier and is described in Subsection~\ref{s-new.4}.
The case of odd $n$ is more complicated and is described in Subsection~\ref{s-new.5}.
As soon as the constructions are presented and the properties described in the
previous paragraph are established, the claim of the proposition follows.
\end{proof}

\subsection{The even case}\label{s-new.4}
For $\mu\neq (1^{n-1})$, we define $\psi(\mu)$ as the partition obtained from $\mu$
by adding one additional part $1$ (i.e. by adding a new node in the first column).
This definition is impossible for $(1^{n-1})$ since the partition $(1^n)$ is forbidden.
Therefore we define $\psi((1^{n-1}))$ to be the partition $(2^{\frac{n}{2}})$.

Consider the dominance order on partitions, with the usual normalization that
the element $(n)$ is the maximum element of $\Lambda_n$ with respect to this order.
Let us fix a linear extension of the dominance order on $\psi(\Lambda^{(4)}_{n-1})$
such that our special partition $(2^{\frac{n}{2}})$ appears as early as possible.
Then, removing the lowest possible removable node from an element of
$\psi(\Lambda^{(4)}_{n-1})$, gives rise to a linear extension
of the dominance order on $\Lambda^{(4)}_{n-1}$.
We write the matrix $N_n$ following these fixed orderings of rows and columns. 
A fairly generic example, for $n=6$, looks as follows:
\begin{displaymath}
\begin{array}{c||c|c|c|c|c|c}
&(1^5)&(2,1^3)&(2^2,1)&(3,1^2)&(3,2)&(4,1)\\ 
\hline\hline
(2,1^4)&{\color{magenta}5}&{\color{magenta}(x-4)}&{\color{magenta}0}&0&0&0\\
\hline
(2^2,1^2)&{\color{magenta}0}&{\color{magenta}3}&{\color{magenta}(x-3)}&0&0&0\\
\hline
(2^3)&{\color{magenta}0}&{\color{magenta}0}&{\color{magenta}1}&0&0&0\\
\hline
(3,1^3)&0&1&0&{\color{violet}(x-3)}&{\color{violet}0}&{\color{violet}0}\\
\hline
(3,2,1)&0&2&2&{\color{violet}0}&{\color{violet}(x-2)}&{\color{violet}0}\\
\hline
(4,1^2)&0&0&0&{\color{violet}1}&{\color{violet}0}&{\color{violet}(x-1)}
\end{array}
\end{displaymath}

From this example we can observe that $N_n$ becomes a lower block-triangular matrix.
Here we have:
\begin{itemize}
\item The {\color{magenta}magenta} square block whose rows are indexed by 
all partitions of the form $(2^i,1^j)$.
\item The {\color{violet}violet} square block whose rows are indexed by the
remaining partitions.
\end{itemize}
To justify this, we note that the partition $(2^{\frac{n}{2}})$ has exactly one removable
node and hence the corresponding row has exactly one non-zero element, which is, moreover, equal
to $1$ and is on the diagonal due to our choice of orderings.

All other partitions of the form $(2^i,1^j)$, where $j>0$, are linearly ordered by the exponent $i$
and have exactly two removable nodes. Removing the removable node in the second column
gives the diagonal coefficient $j+1$, removing the removable node in the first column
gives the coefficient $x-(i+j-1)$ immediately on the right of the diagonal.
This implies the following:
\begin{itemize}
\item The magenta part is upper triangular with non-zero integers on the diagonal.
\item We have only zeros to the right of the magenta part in $N_n$.
\end{itemize}
This proves, in particular, that $N_n$ is lower block-triangular, as claimed.
Also, the determinant of the magenta block is a non-zero integer.

Let us now look at the violet block. Given $\psi(\mu)$ which indexes a row there,
the maximal, with respect to the dominance order, element which can be obtained
from $\psi(\mu)$ by removing a removable node is $\mu$ itself, 
by construction of $\psi(\mu)$. This is because removing nodes from higher rows 
can be interpreted as moving a node from a higher to a lower row, which decreases
the dominance order, by definition. This proves that
the violet block is lower triangular. The diagonal of this block consists if
$x-l$, where $l$ is the height of the first column of the corresponding $\mu$.
In particular, $l<n$. Consequently, up to a non-zero integer, the determinant of
$N_n$ is a product of such factors.

This completes the proof of all necessary properties of $N_n$ in the case of even $n$.

\subsection{The odd case}\label{s-new.5}

For $\mu\neq (1^{n-1})$, we define $\psi(\mu)$ as the partition obtained from $\mu$
by adding one additional part $1$ (i.e. by adding a new node in the first column).
This definition is impossible for $(1^{n-1})$ since the partition $(1^n)$ is forbidden.
Therefore we define $\psi((1^{n-1}))$ to be the partition $(3,2^{\frac{n-3}{2}})$.
We immediately see why this case will be more difficult: the partition 
$(3,2^{\frac{n-3}{2}})$ has two removable nodes and there is no way around this issue
due to the restriction that we work with $\Lambda^{(4)}$ and not $\Lambda$.

Similarly to the even case, we use the dominance order on partitions to order the indexes
for the rows and columns in $N_n$. Just like there, we get a block triangular 
matrix with the magenta block in which the rows are indexed by all partitions
dominated by $(3,2^{\frac{n-3}{2}})$ and the violet block corresponding to the rest.
The properties of the violet block are exactly the same as in the even case.
However, the magenta block is significantly more complex and we  have to consider it
separately. Again, here is a fairly generic example of this magenta block, 
for $n=11$:

\resizebox{\textwidth}{!}{
$
\begin{array}{c||c|c|c|c|c|c|c|c|c|c}
&(1^{10})&(2,1^8)&(2^2,1^6)&(2^3,1^4)&(2^4,1)&(2^5)&(3,1^7)&(3,2,1^5)&(3,2^2,1^3)&(3,2^3,1)\\ 
\hline\hline
(2,1^9)&{\color{red}10}&{\color{teal}(x-9)}&0&0&0&0&0&0&0&0\\
\hline
(2^2,1^7)&0&{\color{red}8}&{\color{teal}(x-8)}&0&0&0&0&0&0&0\\
\hline
(2^3,1^5)&0&0&{\color{red}6}&{\color{teal}(x-7)}&0&0&0&0&0&0\\
\hline
(2^4,1^3)&0&0&0&{\color{red}4}&{\color{teal}(x-6)}&0&0&0&0&0\\
\hline
(2^5,1)&0&0&0&0&{\color{red}2}&{\color{teal}(x-5)}&0&0&0&0\\
\hline
(3,1^8)&0&{\color{blue}1}&0&0&0&0&{\color{teal}(x-8)}&0&0&0\\
\hline
(3,2,1^6)&0&0&{\color{blue}2}&0&0&0&{\color{orange}7}&{\color{teal}(x-7)}&0&0\\
\hline
(3,2^2,1^4)&0&0&0&{\color{blue}3}&0&0&0&{\color{orange}5}&{\color{teal}(x-6)}&0\\
\hline
(3,2^3,1^2)&0&0&0&0&{\color{blue}4}&0&0&0&{\color{orange}3}&{\color{teal}(x-5)}\\
\hline
(3,2^4)&0&0&0&0&0&{\color{blue}5}&0&0&0&{\color{orange}1}\\
\hline
\end{array}
$}

Here we again use color to emphasize different patterns in the matrix:
\begin{itemize}
\item The {\color{teal}teal} color gives the coefficient after removing the
removable node in the first column.
\item The {\color{red}red} color gives the coefficient after removing the
other removable node in a partition of the form $(2,1^j)$.
\item The {\color{blue}blue} color gives the coefficient after removing the
removable node in the first row of a partition which has a part equal to $3$.
\item The {\color{orange}orange} color gives the coefficient after removing the
remaining removable node of a partition of the form $(3,2^i,1^j)$, with $i\neq 0$.
\end{itemize}
We need to compute the determinant of this matrix.
One could immediately note that we can delete the first row and the first column
of the matrix (and record the non-zero factor in the determinant given by the
coefficient $n-1$ in the $((2,1^{n-2}),(1^{n-1}))$-entry).

Now we do the following reduction, illustrated by our specific example:
\begin{itemize}
\item Multiply the column for $(3,2^3,1)$ with $-5$ and add it to the column for
$(2^5)$. This will kill the entry $5$ in the last row.
\item Take out the common multiple $(x-5)$ from the column of $(2^5)$.
This will make the $((2^5,1),(2^5))$-entry equal to $1$.
\item Multiply the column for $(2^5)$ with $-2$ and add it to the column
of $(2^4,1)$. This will kill the entry $2$ in the row of $(2^5,1)$.
\end{itemize}
Now we have a unique non-zero entry in the row of $(2^5,1)$, so we can delete this
row and the column for $(2^5)$. The obtained matrix will be of the same form
as the original one, but of smaller size. The only major difference is the 
coefficient in the entry $((3,2^3,1^2),(2^4,1))$. It is equal to $4+(-2)(-5)$,
which is a sum of products of some non-zero entries of the original matrix.

Proceeding recursively in the same manner, we obtain that the determinant of 
our matrix equals the expression $(x-5)(x-6)(x-7)(x-8)$, up to a non-zero integer factor.
The factor is non-zero as it is a non-trivial sum of products of the entries of the 
original matrix. Note that no signs appear (i.e. the two minus signs in the product
$(-2)(-5)$ above cancel each other)!

The whole procedure as described above, clearly, works for an 
arbitrary odd $n$ in a similar way. In particular, 
it follows that, up to a non-zero integer, the determinant of
$N_n$ is a product of the desired factors.

This completes the proof of all necessary properties of $N_n$ in the case of odd $n$.

\subsection{Completing the proof of Lemma~\ref{lem:ker}}\label{s-new.9}

We want to show that the adjoint $\g$-submodule of $U(\rr)$ generated by
all elements of the form $v_{-4}^iv_{-2}^{k-i}$, where $i=0,1,\dots,k$,
is isomorphic to the $\g$-module given by the decomposition \eqref{eqn1}.
Another way to put it is as follows: we need to show that the elements
$\mathrm{ad}_e^{i}(v_{-4}^iv_{-2}^{k-i})$, for $i=0,1,\dots,k-1$,
(note: the value $k$ is gone!) are linearly independent, because this is
what we need to generate all $k$ summands in the decomposition \eqref{eqn1}.

It is convenient to rescale the basis $\{v_{-4},v_{-2},v_0,v_2,v_4\}$
to a new basis $\{w_{-4},w_{-2},w_0,w_2,w_4\}$ such that 
\begin{displaymath}
[e,w_{-4}]=w_{-2},\quad
[e,w_{-2}]=w_{0},\quad
[e,w_{0}]=w_{2},\quad
[e,w_{2}]=w_{4},\quad
[e,w_4]=0.
\end{displaymath}
With this convention, we identify the monomial $w_{-4}^aw_{-2}^bw_{0}^cw_{2}^dw_{4}^r$
with the partition $\lambda\in\Lambda^{(4)}$ which has
\begin{itemize}
\item $b$ parts equal to $1$, then $c$ parts equal to $2$, then $d$ parts equal to $3$,
and, finally, $r$ parts equal to $4$.
\end{itemize}
We note that the exponent $a$ is not used. However, since $k$ is fixed, 
we have $a=k-b-c-d-r$. Note that $l=b+c+d+r$ is exactly the height of the first column 
of $\lambda$ and hence $a$ is the evaluation of the corresponding coefficient
$x-l$ at $x=k$.

Since $[e,{}_-]$ is a derivation, its application to $w_{-4}^aw_{-2}^bw_{0}^cw_{2}^dw_{4}^r$
results in a linear combination of monomials corresponding to $\mu$ such that there is
an edge $\lambda\to \mu$ in $\Lambda^{(4)}$ and the coefficient is given
exactly by the label of the corresponding edge (under the evaluation of $x$ to $k$, whenever 
appropriate).  Therefore the fact that the elements
$\mathrm{ad}_e^{i}(v_{-4}^iv_{-2}^{k-i})$, for $i=0,1,\dots,k-1$,
are linearly independent, translates exactly into the claim of
Proposition~\ref{prop-nnn1}.

\section{Some generating functions related to symmetric powers of simple $\mathfrak{sl}_2$-modules}
\label{sec:gen_funct}

Here we collect technical computations which are elementary, and independent of previous sections. Algebraic manipulations with rational functions were computer assisted with SageMath v.8.9.

For $k,n,l \in \Z_{\geq 0}$ we denote by $\Sym^n(L(k))_l$ the $l$-weight space of the $n$-th symmetric power of $L(k)$. Define the following generating function:
\[ F^{(k)}_l(q) := \sum_{n \geq 0} \dim\left( \Sym^n(L(k))_l \right) \cdot q^n. \]
From these generating functions one can extract multiplicities of simple $\mathfrak{sl}_2$-modules in $\Sym L(k)$. Namely,
\begin{equation}
F^{(k)}_l(q) - F^{(k)}_{l+2}(q) = \sum_{n \geq 0} [\Sym^n(L(k)) \colon L(l) ] \cdot q^n.
\end{equation}
In particular, $F^{(k)}_0(q) - F^{(k)}_{2}(q)$ gives a lot of information on the invariant algebra $(\Sym L(k))^\g$. It is easy to see that $(\Sym L(k))^\g$ is a polynomial algebra if and only if $F^{(k)}_0(q) - F^{(k)}_{2}(q)$ is a product of functions of the form $\frac{1}{1-q^i}$, and that each $i$ appearing in the product gives a generator of degree $i$. 

Denote by $\mathbf{m}$ the ideal in $\Sym L(k)$ generated by the invariants $(\Sym L(k))^\g$. The quotient $Q := \quotient{\Sym L(k)}{\mathbf{m}}$ is also graded by the degree. Denote its graded components by $Q^n$. The following proposition follows from the definitions:
\begin{proposition}\label{proposition:freeness_series}
Assume that $\Sym L(k)$ is free over the invariants $(\Sym L(k))^\g$. Then
\begin{equation}
\frac{\displaystyle\sum_{n \geq 0} [\Sym^n(L(k)) \colon L(l) ] \cdot q^n}{
\displaystyle\sum_{n \geq 0} [\Sym^n(L(k)) \colon L(0) ] \cdot q^n} = 
\sum_{n \geq 0} [Q^n \colon L(l) ] \cdot q^n.
\end{equation}
In particular, in this case for any $l \geq 0$ the power series of $\displaystyle\frac{F^{(k)}_l(q) - F^{(k)}_{l+2}(q)}{F^{(k)}_0(q) - F^{(k)}_{2}(q)}$ has only non-negative coefficients.
\end{proposition}

In this section we will calculate $F^{(k)}_l(q)$ in a closed form as a rational function, for $k$ up to $6$. The first step is clear:

\begin{lemma} \label{lemma:F_diophant}
We have
\begin{equation} \label{equation:sum_diophant}
F^{(k)}_l(q) = \sum q^{a_0 + a_1 + \ldots + a_k},
\end{equation}
where the sum is taken over all $(k+1)$-tuples of non-negative integers $(a_0,\ldots,a_k)$ satisfying the Diophantine equation
\begin{equation} \label{equation:diophant}
k \cdot a_0 + (k-2) \cdot a_1 + (k-4) \cdot a_2 + \ldots + (-k+2) \cdot a_{k-1} + (-k) \cdot a_{k} = l.
\end{equation}
\end{lemma}
\begin{proof}
Follows from considering the monomial basis of $\Sym L(k)$ consisting of products of the weight vectors $v_i$ in $L(k)$. 
\end{proof}

\begin{lemma} \label{lemma:F012}
We have
\begin{align*}
&F^{(0)}_l(q) = \begin{cases}\displaystyle \frac{1}{1-q} & \colon l=0, \\ 0 & \colon l > 0, \end{cases} &  &F^{(1)}_l(q) = \frac{q^l}{1-q^2},
\end{align*}
\begin{equation*}
F^{(2)}_l(q) = \begin{cases}\displaystyle \frac{q^{\frac{l}{2}}}{(1-q)(1-q^2)} & \colon \text{$l$ even}, \\ 0 & \colon \text{$l$ odd}. \end{cases}
\end{equation*}
\end{lemma}
\begin{proof}
Easy calculation using Lemma \ref{lemma:F_diophant}.
\end{proof}

The following lemma allows somewhat inductive approach in calculating $F_l^{(k)}$.
\begin{lemma} \label{lemma:F_recursion}
For $k \geq 2$ we have
\begin{equation} \label{equation:F_recursion}
F^{(k)}_l(q) = \sum_{\substack{k a +b - kc=l \\ a,b,c, \geq 0}} q^{a+c} F^{(k-2)}_{b}(q) \ + \ \sum_{\substack{k a -b - kc=l+1 \\ a,b,c, \geq 0}} q^{a+c} F^{(k-2)}_{b+1}(q)
\end{equation}
\end{lemma}
\begin{proof}
In the equation (\ref{equation:diophant}) denote $a=a_0$, $c=a_k$ and the middle part $b := (k-2)a_1+ \ldots +(-k+2)a_{k-1}$, and separate (\ref{equation:sum_diophant}) into two sums depending on whether $b \geq 0$ or $b<0$. In the latter sum rename $b$ to $-b-1$. This gives (\ref{equation:F_recursion})
\end{proof}

\begin{proposition} \label{proposition:F3}
We have
\begin{align}
\label{item:F3_1}  & F^{(3)}_l(q) = \begin{cases}\displaystyle
\frac{q^{\frac{l}{3}}(1 + q^2+ q^4 -q^{\frac{2l}{3}+2})}{(1-q^2)^2(1-q^4)} & \colon l \equiv 0 \text{ (mod $3$)}, \\[15pt] \displaystyle
\frac{q^{\frac{l+2}{3}}(1+2q^2-q^{\frac{2l+4}{3}})}{(1-q^2)^2(1-q^4)} & \colon l \equiv 1 \text{ (mod $3$)}, \\[15pt] \displaystyle
\frac{q^{\frac{l+4}{3}}(2+q^2-q^{\frac{2l+2}{3}})}{(1-q^2)^2(1-q^4)} & \colon l \equiv 2 \text{ (mod $3$)},
\end{cases} \\[1em]
\label{item:F3_2}  & F^{(3)}_0(q) - F^{(3)}_2(q) = \frac{1}{1-q^4} = 1 + q^4 + q^8 + \ldots, \\[1em]
\label{item:F3_3}  & \frac{F^{(3)}_l(q) - F^{(3)}_{l+2}(q)}{F^{(3)}_0(q) - F^{(3)}_2(q)} = \begin{cases}\displaystyle
q^{\frac{l}{3}} + q^{\frac{l}{3}+2} + q^{\frac{l}{3}+4} + \ldots + q^{l} & \colon l \equiv 0 \text{ (mod $3$)}, \\
q^{\frac{l+8}{3}} + q^{\frac{l+8}{3}+2} + \ldots + q^{l} & \colon l \equiv 1 \text{ (mod $3$)}, \\
q^{\frac{l+4}{3}} + q^{\frac{l+4}{3}+2} + \ldots + q^{l} & \colon l \equiv 2 \text{ (mod $3$)}.
\end{cases}
\end{align}

\end{proposition}
\begin{proof}
Claims (\ref{item:F3_2}) and (\ref{item:F3_3}) follow from (\ref{item:F3_1}) by direct calculations. To prove (\ref{item:F3_1}), we use Lemma \ref{lemma:F_recursion} and Lemma \ref{lemma:F012}: 
\begin{equation} \label{equation:recursion3}
F^{(3)}_l(q) = \frac{1}{1-q^2}\Bigg(\underbrace{\sum_{\substack{3 a +b - 3c=l \\ a,b,c, \geq 0}} q^{a+b+c}}_{A} \ + \ q \cdot  \underbrace{\sum_{\substack{3 a -b - 3c=l+1 \\ a,b,c, \geq 0}} q^{a+b+c}}_{B} \Bigg).
\end{equation}
Let us consider $A$ first. The equation $3 a +b - 3c=l$ implies that $b \equiv l$ (mod $3$); put $b=3b'+r$ and $l=3l'+r$ for $r \in \{0,1,2\}$. We have
\begin{align*}
A &=  q^r \cdot \sum_{a +b' -c=l'} q^{a+3b'+c} \\
&= q^r \cdot \sum_{b' \leq c+l'} q^{(l'+c-b')+3b'+c} \\
&= q^{l'+r} \cdot  \sum_{c \geq 0} \left( q^{2c} \cdot \sum_{b'=0}^{c+l'} q^{2b'} \right) \\
&= \frac{q^{l'+r}}{1-q^2} \cdot  \sum_{c \geq 0}  q^{2c} \cdot \left( 1- q^{2c+2l'+2} \right) \\
&= \frac{q^{l'+r}}{1-q^2} \cdot \left( \frac{1}{1-q^2}-     \frac{q^{2l'+2}}{1-q^4} \right). \\
\end{align*}
Now we switch to $B$. Here it is more convenient to work out each case $l$ (mod $3$) separately. So assume $l \equiv 0$ (mod $3$), the other cases being very similar. Put $l=3l'$.  The equation  $3 a -b - 3c=l+1$ implies that $b=3b'+2$, so
\begin{align*}
B &=  q^2 \cdot \sum_{a - b' - c=l'+1} q^{a+3b'+c} \\
&= q^2 \cdot \sum_{b',c \geq 0} q^{(b'+c+l'+1)+3b'+c} \\
&= q^{l'+3} \cdot \sum_{b',c \geq 0} q^{4b'+2c} \\
&= \frac{q^{l'+3}}{(1-q^2)(1-q^4)}.
\end{align*}
Plugging $A$ and $B$ into (\ref{equation:recursion3}), after some tiding up, gives the first case in (\ref{item:F3_1}). The other two cases follow similarly.
\end{proof}

\begin{proposition}[C. Krattenthaler] \label{proposition:F4}
We have
\begin{align}
\label{item:F4_1} & F^{(4)}_l(q) = \begin{cases}\displaystyle
\frac{q^{\frac{l}{4}}(1+q^2-q^{\frac{l}{4}+1})}{(1-q)^2 (1-q^2)(1-q^3)} & \colon l \equiv 0 \text{ (mod $4$)}, \\[15pt] \displaystyle
\frac{q^{\frac{l+2}{4}}(1+q-q^{\frac{l+2}{4}+1})}{(1-q)^2 (1-q^2)(1-q^3)} & \colon l \equiv 2 \text{ (mod $4$)}, \\[15pt] \displaystyle
0 & \colon l \text{ odd},
\end{cases} \\[1em]
\label{item:F4_2} & F^{(4)}_0(q) - F^{(4)}_2(q) = \frac{1}{(1-q^2)(1-q^3)}, \\[1em]
\label{item:F4_3} & \frac{F^{(4)}_l(q) - F^{(4)}_{l+2}(q)}{F^{(4)}_0(q) - F^{(4)}_2(q)} = \begin{cases}\displaystyle
q^{\frac{l}{4}} + q^{\frac{l}{4}+1} + q^{\frac{l}{4}+2} + \ldots + q^{\frac{l}{2}} & \colon l \equiv 0 \text{ (mod $4$)}, \\
q^{\frac{l+6}{4}} + q^{\frac{l+6}{4}+1} + \ldots + q^{\frac{l}{2}} & \colon l \equiv 2 \text{ (mod $4$), } l \neq 2, \\
0 & \colon l \text{ odd or } l=2.
\end{cases} 
\end{align}
\end{proposition}
\begin{proof}
Similar to the proof of Proposition \ref{proposition:F3}. One uses Lemma \ref{lemma:F_recursion} to get an expression involving $F^{(2)}_l$'s and Lemma \ref{lemma:F012}  to plug in closed expression for these $F^{(2)}_l$'s.
\end{proof}

\begin{proposition}[S. Wagner] \label{proposition:F5}
We have
\begin{align}
\nonumber & F^{(5)}_0(q) = \frac{1 + q^2 + 6 q^4 + 9 q^6 + 12 q^8 + 9 q^{10} + 6 q^{12} + q^{14} + q^{16}}{(1 - q^2)^2 (1 - q^4) (1 - q^6) (1 - q^8)  }, \\[1em]
\nonumber & F^{(5)}_1(q) = \frac{q (1 + 3 q^2 + 5 q^4 + 5 q^6 + 5 q^8 + 3 q^{10} + q^{12})}{(1 - q^2)^3 (1 - q^6) (1 - q^8)} ,  \\[1em]
\nonumber & F^{(5)}_2(q) = \frac{q^2 (3 + 5 q^2 + 7 q^4 + 5 q^6 + 3 q^8)}{(1 - q^2)^2 (1 - q^4)^2 (1 - q^6)} ,  \\[1em]
\nonumber & F^{(5)}_3(q) = \frac{q (1 + 3 q^2 + 4 q^4 + 7 q^6 + 4 q^8 + 3 q^{10} + q^{12})}{(1 - q^2)^3 (1 - q^6) (1 - q^8)} , \\[1em]
\label{item:F5_2} & F^{(5)}_0(q) - F^{(5)}_2(q) = \frac{1-q^{36}}{(1-q^4)(1-q^8)(1-q^{12})(1-q^{18})}, \\[1em]
\label{item:F5_3} & \frac{F^{(5)}_1(q) - F^{(5)}_{3}(q)}{F^{(5)}_0(q) - F^{(5)}_2(q)} = \frac{q^5(1+q^2)}{1-q^6+q^{12}} .
\end{align}
The rational function \eqref{item:F5_3}  is not a polynomial, and its power series contains a negative coefficient next to $q^{23}$.
\end{proposition}
\begin{proof}
Similar to the proof of Proposition \ref{proposition:F3}, but with considerably more work. One uses Lemma \ref{lemma:F_recursion} to obtain (\ref{equation:F_recursion}), and then one can use (\ref{item:F3_1}). For this, one should split each of the two sums in (\ref{equation:F_recursion}) further into three sums, depending on $b$ (mod $3$). Calculating these six sums requires also further splitting into cases depending on $b$ (mod $5$).
\end{proof}

\begin{proposition} \label{proposition:F6}
We have
\begin{align}
\nonumber & F^{(6)}_0(q) =  \frac{ 1+ q^2+ 3q^3+ 4q^4+ 4q^5+ 4q^6  + 3q^7+ q^8+q^{10} }{(1-q)(1-q^2)^2(1-q^3)(1-q^4)(1-q^5)}  ,  \\[1em]
\nonumber & F^{(6)}_2(q) = \frac{q( 1+ 2q+ 2q^2+ q^3 + 2q^4 + 2q^5+q^6 )}{ (1-q)(1-q^2)^3(1-q^3)(1-q^5) } ,   \\[1em]
\nonumber & F^{(6)}_4(q) =  \frac{q(1+ 2q+ 2q^2+ 4q^3 + 4q^4 + 4q^5+ 2q^6 + 2q^7+q^8 ) }{ (1-q)(1-q^2)^2(1-q^3)(1-q^4)(1-q^5) }   , \\[1em]
\nonumber & F^{(6)}_6(q) =  \frac{q(1+ q+ 2q^2+ 3q^3+ 2q^4 + q^5+q^6) }{ (1-q)(1-q^2)^3(1-q^3)(1-q^5)} , \\[1em]
\label{item:F6_2} & F^{(6)}_0(q) - F^{(6)}_2(q) = \frac{1-q^{30}}{ (1-q^2)(1-q^4)(1-q^6)(1-q^{10})(1-q^{15})}, \\[1em]
\label{item:F6_3} & \frac{F^{(6)}_2(q) - F^{(6)}_{4}(q)}{F^{(6)}_0(q) - F^{(6)}_2(q)} =  \frac{q^3( 1 + q + q^2 ) }{1+ q- q^3- q^4- q^5+ q^7 +q^8  } .
\end{align}
The rational function \eqref{item:F6_3} is not a polynomial, and its power series contains a negative coefficient next to $q^{18}$.
\end{proposition}
\begin{proof}
Similar to the proof of Proposition \ref{proposition:F3}, but with considerably more work. One uses Lemma \ref{lemma:F_recursion} to obtain (\ref{equation:F_recursion}), and then one can use (\ref{item:F4_1}). For this, one should split each of the two sums in (\ref{equation:F_recursion}) further into two sums, depending on $b$ (mod $4$) (the odd case being trivial). Calculating these two sums requires also further splitting into cases depending on $b$ (mod $6$).
\end{proof}

\noindent
V.~M.: Department of Mathematics, Uppsala University, Box. 480,
SE-75106, Uppsala, SWEDEN, email: {\tt mazor\symbol{64}math.uu.se}

\noindent
R.~M.: Department of Mathematics, Uppsala University, Box. 480,
SE-75106, Uppsala, SWEDEN, email: {\tt rafaelmrdjen\symbol{64}gmail.com}


\begin{thebibliography}{99999999}

\bibitem[AIM19]{alshammari2019on}
F.~Alshammari, P.~S. Isaac, and I.~Marquette.
\newblock On {C}asimir operators of conformal {G}alilei algebras.
\newblock {\em J. Math. Phys.}, 60(1):013509, 14, 2019.

\bibitem[AS03]{andersen2003twisting}
H.~H. Andersen and C.~Stroppel.
\newblock Twisting functors on {$\mathcal{O}$}.
\newblock {\em Represent. Theory}, 7:681--699, 2003.

\bibitem[Ark04]{arkhipov2004algebraic}
S.~Arkhipov.
\newblock Algebraic construction of contragradient quasi-{V}erma modules in
  positive characteristic.
\newblock In {\em Representation theory of algebraic groups and quantum
  groups}, volume~40 of {\em Adv. Stud. Pure Math.}, pages 27--68. Math. Soc.
  Japan, Tokyo, 2004.
 
\bibitem[Deo80]{deodhar1980on}
V.~V. Deodhar.
\newblock On a construction of representations and a problem of {E}nright.
\newblock {\em Invent. Math.}, 57(2):101--118, 1980.
 
\bibitem[DLMZ14]{dubsky2014category}
B.~Dubsky, R.~L\"{u}, V.~Mazorchuk, and K.~Zhao.
\newblock Category {$\mathcal{O}$} for the {S}chr\"{o}dinger algebra.
\newblock {\em Linear Algebra Appl.}, 460:17--50, 2014.

\bibitem[Enr79]{enright1979on}
T.~J. Enright.
\newblock On the fundamental series of a real semi-simple {L}ie algebra: their
  irreducibility, resolutions and multiplicity formulae.
\newblock {\em Ann. of Math. (2)}, 110(1):1--82, 1979.

\bibitem[FO05]{futorny2005kostant}
V.~Futorny and S.~Ovsienko.
\newblock Kostant's theorem for special filtered algebras.
\newblock {\em Bull. London Math. Soc.}, 37(2):187--199, 2005.

\bibitem[GK12]{gomis2012schrodinger}
J.~Gomis and K.~Kamimura.
\newblock Schr\"{o}dinger equations for higher order nonrelativistic particles
  and $n$-{G}alilean conformal symmetry.
\newblock {\em Phys. Rev. D}, 85:045023, Feb 2012.

\bibitem[HHLS18]{hahn2018from}
H.~Hahn, J.~Huh, E.~Lim, and J.~Sohn.
\newblock From partition identities to a combinatorial approach to explicit
  {S}atake inversion.
\newblock {\em Ann. Comb.}, 22(3):543--562, 2018.

\bibitem[Hil93]{hilbert1993theory}
D.~Hilbert.
\newblock {\em Theory of algebraic invariants}.
\newblock Cambridge University Press, Cambridge, 1993.
\newblock Translated from the German and with a preface by Reinhard C.
  Laubenbacher, Edited and with an introduction by Bernd Sturmfels.

\bibitem[Hum72]{humphreys1978introduction}
J.~E. Humphreys.
\newblock {\em Introduction to {L}ie algebras and representation theory}.
\newblock Springer-Verlag, New York-Berlin, 1972.
\newblock Graduate Texts in Mathematics, Vol. 9.

\bibitem[Hum08]{humphreys2008representations}
J.~E. Humphreys.
\newblock {\em Representations of semi-simple {L}ie algebras in the {BGG}
  category {$\mathcal{O}$}}, volume~94 of {\em Graduate Studies in
  Mathematics}.
\newblock American Mathematical Society, Providence, RI, 2008.

\bibitem[Kna88]{knapp1988lie}
A.~W. Knapp.
\newblock {\em Lie groups, {L}ie algebras, and cohomology}, volume~34 of {\em
  Mathematical Notes}.
\newblock Princeton University Press, Princeton, NJ, 1988.

\bibitem[KM05]{khomenko2005on}
O.~Khomenko and V.~Mazorchuk.
\newblock On {A}rkhipov's and {E}nright's functors.
\newblock {\em Math. Z.}, 249(2):357--386, 2005.

\bibitem[KM02]{konig2002enrights}
S.~K\"{o}nig and V.~Mazorchuk.
\newblock {E}nright's completions and injectively copresented modules.
\newblock {\em Trans. Amer. Math. Soc.}, 354(7):2725--2743, 2002.

\bibitem[Kra15]{krause2015krull}
H.~Krause.
\newblock {K}rull-{S}chmidt categories and projective covers.
\newblock {\em Expo. Math.}, 33(4):535--549, 2015.

\bibitem[LMZ14]{lu2014simple}
R.~L\"{u}, V.~Mazorchuk, and K.~Zhao.
\newblock On simple modules over conformal {G}alilei algebras.
\newblock {\em J. Pure Appl. Algebra}, 218(10):1885--1899, 2014.

\bibitem[Mat00]{mathieu2000classification}
O.~Mathieu.
\newblock Classification of irreducible weight modules.
\newblock {\em Ann. Inst. Fourier (Grenoble)}, 50(2):537--592, 2000.

\bibitem[Maz10]{mazorchuk2010lectures}
V.~Mazorchuk,
\newblock {\em Lectures on {$\mathfrak{sl}_2(\mathbb{C})$}-modules},
\newblock Imperial College Press, London, 2010.

\bibitem[MM20]{mazorchuk2020lie}
V.~Mazorchuk and R.~Mr{\dj}en.
\newblock {L}ie algebra modules which are locally finite over the semi-simple
  part, 2020.
\newblock arXiv:2001.02967 [math.RT], to appear in {\em Nagoya Math. J.}


\bibitem[MS19]{mazorchuk2019category}
V.~Mazorchuk and C.~S\"{o}derberg.
\newblock Category $\mathcal{O}$ for {T}akiff $\mathfrak{sl}_2$.
\newblock {\em J. Math. Phys.}, 60(11):111702, 2019.

\bibitem[Sa01]{sagan}
B.~Sagan, 
\newblock {\em The symmetric group. Representations, combinatorial algorithms, and symmetric functions.} 
\newblock Second edition. Graduate Texts in Mathematics, 203. Springer-Verlag, New York, 
2001.
\end{thebibliography}
\end{document}